\documentclass[11pt]{amsart}
\usepackage{palatino}
\usepackage{amsfonts}
\usepackage{amssymb}
\usepackage{amscd}
\usepackage[dvips]{graphicx}
\input xy
\xyoption{all}
\usepackage[english]{babel}
\usepackage[all]{xy}
\usepackage[breaklinks,bookmarksopen,bookmarksnumbered]{hyperref}

\newtheorem{theorem}{Theorem}[section]
\newtheorem{proposition}[theorem]{Proposition}
\newtheorem{corollary}[theorem]{Corollary}
\newtheorem{lemma}[theorem]{Lemma}
\theoremstyle{definition}
\newtheorem{definition}[theorem]{Definition}
\newtheorem{remark}[theorem]{Remark}
\newtheorem{example}[theorem]{Example}

\newtheorem{conjecture/question}[theorem]{Conjecture/Question}

\newtheorem{remark/definition}[theorem]{Remark/Definition}
\newtheorem{terminology/notation}[theorem]{Terminology/Notation}

\newcommand{\marginlabel}[1]%
  {\mbox{}\marginpar{\raggedleft\hspace{0pt}\bfseries\sf#1}}

\def\FF{{\mathbb F}}
\def\ZZ{{\mathbb Z}}
\def\ZZ{{\mathbb Z}}
\def\GG{{\textbf G}}
\def\QQ{{\mathbb Q}}
\def\PP{{\textbf P}}
\def\OO{\mathcal{O}}

\def\cD{\mathcal{D}}
\def\cB{\mathcal{B}}
\def\cA{\mathcal{A}}
\def\F{\mathcal{F}}
\def\P{\mathcal{P}}

\def\E{\mathcal{E}}
\def\G{\mathcal{G}}

\def\I{\mathcal{I}}

\def\cM{\mathcal{M}}
\def\cR{\mathcal{R}}
\def\rr{\overline{\mathcal{R}}}
\def\cZ{\mathcal{Z}}
\def\cU{\mathcal{U}}

\def\mm{\overline{\mathcal{M}}}

\def\zz{\overline{\mathcal{Z}}}

\def\rem{\overline{\textbf{M}}}

\def\pem{\widetilde{\textbf{M}}}

\def\ssR{\mathsf{R}}

\def\ssC{\mathsf{C}}

\def\ol{\overline}
\def\wt{\widetilde}

\def\mbb{\mathbb}

\def\rem{\ol{\mathsf{M}}}
\def\ssrr{\ol{\mathsf{R}}}

\def\pem{\mbb{\mathsf{M}}}

\newcommand{\tensor}{\otimes}
\newcommand{\isom}{\cong}
\newcommand{\codim}{\operatorname{codim}}

\newcommand{\lideal}{\langle}
\newcommand{\rideal}{\rangle}

\newcommand{\Tor}{\operatorname{Tor}}
\newcommand{\Spec}{\operatorname{Spec}}

\newcommand{\ch}{\operatorname{char}}

\newcommand{\SL}{\operatorname{SL}}

\newcommand{\Pic}{\operatorname{Pic}}
\newcommand{\sPic}{{\mathsf{Pic}}}

\newcommand{\Aa}{{\mathbb A}}
\newcommand{\sHom}{\operatorname{\mathcal Hom}} 
\newcommand{\sExt}{\operatorname{\mathcal Ext}} 
\newcommand{\Hom}{\operatorname{Hom}} 
\newcommand{\Ext}{\operatorname{Ext}} 
\newcommand{\Sym}{\operatorname{Sym}} 

\begin{document}
\title{Syzygies of torsion bundles and the geometry  of the level $\ell$ modular variety over $\mm_g$}

\author[A.~Chiodo]{Alessandro Chiodo}
\address{Universit{\'e} de Grenoble,
U.M.R.~CNRS 5582,
Institut Fourier,
\newline \indent
BP 74, 38402
Saint Martin d'Hères,
France }
 \email{{\tt
chiodo@ujf-grenoble.fr}}

\author[D. Eisenbud]{David Eisenbud}

\address{University of California, Department of Mathematics, Berkeley, CA 94720, USA}
\email{{\tt eisenbud@math.berkeley.edu}}

\author[G. Farkas]{Gavril Farkas}

\address{Humboldt Universit\"at zu Berlin, Institut f\"ur Mathematik, Unter den Linden 6,\newline
\indent
10099 Berlin, Germany }
\email{{\tt farkas@math.hu-berlin.de}}
\thanks{}

\author[F.-O. Schreyer]{Frank-Olaf Schreyer}
\address{Uversit\"at des Saarlandes, Campus E2 4, D-66123 Saarbr\"ucken, Germany
}
 \email{{\tt schreyer@math.uni-sb.de}}

\markboth{A. CHIODO, D. EISENBUD, G. FARKAS and F.-O. SCHREYER}
{The Kodaira dimension of the Prym moduli space of level three}
\maketitle

\begin{abstract}

We formulate, and in some cases prove, three statements concerning the purity
or, more generally, the \emph{naturality} of the resolution of various modules one can attach to a generic curve of genus g and a torsion point of $\ell$ in its Jacobian. These statements can be viewed an analogues of Green's Conjecture and we verify them computationally for bounded genus. We then compute the cohomology class of the corresponding non-vanishing locus in the moduli space $\cR_{g,\ell}$ of twisted level $\ell$ curves of genus g and use this to derive results about the birational geometry of $\cR_{g, \ell}$. For instance, we prove that $\cR_{g,3}$ is a variety of general type when $g>11$ and the Kodaira dimension of $\cR_{11,3}$ is greater than or equal to 19. In the last section we explain probabilistically the unexpected failure of the Prym-Green conjecture in genus 8 and level 2.
\end{abstract}


One way of proving that the moduli space of curves of odd genus $g\gg 0$ has
general type is via the divisor of curves whose canonical ring has ``extra'' syzygies. In this paper we apply the same philosophy to prove that, in certain cases, the moduli spaces of curves with torsion bundles---the modular varieties of the title---are also of general type. These modular varieties are natural generalizations, in higher genus,
of the much studied modular curves
$X_1(\ell):=\mathcal{H}/\Gamma_1(\ell)$ classifying elliptic curves together with an $\ell$-torsion point in their Jacobian.

To explain what we mean by ``extra'' syzygies, consider a finitely generated graded
module $M$ over a polynomial ring $S = \mathbb C[x_{0}, \dots,x_{n}]$.
Such a module has a minimal  free resolution of the form
$$
0\leftarrow M \leftarrow F_{0}\leftarrow\cdots \leftarrow F_{i}\leftarrow \cdots,
$$
where $F_{i} = \sum_{j}S(-i-j)^{b_{i,j}}$. The numbers $b_{i,j} = b_{i,j}(M)$,
called the graded Betti numbers of $M$, are
uniquely defined; in fact $b_{i,j}$ is the dimension of the degree $i+j$
component of ${\rm Tor}^{S}_{i}(M, \mathbb C)$. We say that the
resolution of $M$ is \emph{natural} if, for each $j$, the number $b_{i,j-i}(M)$
is nonzero for at most one value of $i$, that is, at most one Betti number on each diagonal of the
Betti diagram of $M$ is non-zero. In an irreducible flat family of modules $M_{\lambda}$ the
$b_{i,j}(M_{\lambda})$ are semicontinuous, and simultaneously take
on minimum values on an open set.
We say that $M_{\lambda}$ has ``extra'' syzygies
when one of the values $b_{i,j}(M_{\lambda})$ is larger than
this minimum. If the resolution of some $M_{\lambda}$ is
natural in the sense above, then its Betti numbers have the minimum value.

\vskip 4pt
For example, given a curve $C$, a line bundle $L\in \mbox{Pic}(C)$ and a sheaf $\F$ on $C$, we consider the
$S:=\mbox{Sym } H^0(C, L)$-module
$$
\Gamma_C(\F, L):=\bigoplus_{q\in \ZZ} H^0(C, \F\otimes L^{\otimes q}).
$$
Following notation introduced by Mark Green,
the vector space ${\rm Tor}_{i}^{S}(\Gamma_{C}(\F,L), \mathbb C)_{i+j}$ is often
written $K_{i,j}(C;\F,L)$. Green's Conjecture \cite{G} for generic curves of genus $g$ (proved in \cite{V}), asserts that the resolution of the
canonical ring of $C$ or, equivalently, the resolution of
$\Gamma_{C}(\OO_{C},K_{C})$, is natural. For odd genus, it follows that the resolution is pure, which in Green's notation says
\begin{equation}\label{green}
K_{\frac{g-1}{2}, 1}(C,\mathcal O_{C}, K_C)=0.
\end{equation}
The locus of curves
whose canonical ring has extra syzygies is a divisor whose support
is the Hurwitz locus $\cM_{g, \frac{g+1}{2}}^1$ used in \cite{HM} to prove that these
moduli spaces have general type.

Generalizing the case above, we study modules of the form $\Gamma_{C}(\xi, K_C\otimes \eta)$ where
$\eta$ and $\xi$ are line bundles of degree $0$ on a smooth curve $C$. As a first step, we prove:

\begin{theorem}\label{general bundles}
Let $C$ be a general curve of genus $g$, $\eta\in \mathrm{Pic}^0(C)[\ell]$ a torsion bundle of order $\ell\geq 2$ and $\xi\in \mathrm{Pic}^0(C)$ a general line bundle of degree
degree $0$. Then the module $\Gamma_{C}(\xi, K_C\otimes \eta)$ has natural resolution.
\end{theorem}

We shall use a refinement of this result to compute the Kodaira dimension of the
moduli spaces $\cR_{g, \ell}$ parametrizing level $\ell$ curves $[C, \eta]$, where $C$ is a smooth curve of genus $g$ and $\eta\in \mbox{Pic}^0(C)$ is a torsion line bundle of order $\ell$.
In particular, we shall prove:

\begin{theorem}\label{koddim}
$\cR_{g, 3}$ is a variety of general type for $g\geq 12$. Furthermore, the Kodaira dimension of $\cR_{11, 3}$ is at least $19$.
\end{theorem}

It is known that $\cR_{g, 3}$ is rational for $g\leq 4$, see \cite{BC}, \cite{BV}. The level $\ell$ modular variety that has received most attention so far is the moduli space $\cR_{g, 2}$ classifying Prym varieties of dimension $g-1$.
It is shown in \cite{FL} that $\cR_{g, 2}$ is a variety of general type for $g>13$, whereas $\cR_{g, 2}$ is unirational for $g\leq 7$, see \cite{FV2} and references therein.
\vskip 3pt
We now explain the statements concerning the naturality of the resolution of certain modules one associates to a level curve $[C, \eta]\in \cR_{g, \ell}$. Such predictions can be made for any $g$ and, just like for Green's Conjecture, in about half of the cases they amount to saying that the resolution is actually \emph{pure}. In such case, the locus of points $[C, \eta]$ where the corresponding resolution is not pure is a \emph{virtual} divisor on $\cR_{g, \ell}$, that is, the degeneracy locus of a morphism between vector bundles of the same rank over the stack $\ssR_{g, \ell}$ which is coarsely represented by $\cR_{g, \ell}$. Proving the corresponding syzygy conjecture amounts to showing that the respective degeneracy locus is a \emph{genuine} divisor on $\cR_{g, \ell}$.

We fix  $[C, \eta]\in \cR_{g, \ell}$ and let $L:=K_C\otimes\eta$ be the paracanonical line bundle.
\vskip 3pt
\noindent {\bf{A. Prym-Green Conjecture.}} \emph{For a general level $\ell$ curve  $[C,\eta]\in \cR_{g, \ell}$ of genus $g\ge 6$, the  homogeneous coordinate ring of the paracanonical curve $\phi_{K_C\otimes \eta}:C \hookrightarrow  \PP^{g-2}$ has a natural resolution. Equivalently, in even genus, the resolution is pure and the paracanonical curve satisfies property $(N_{\frac{g}{2}-3})$, that is,
\begin{equation}\label{prymgreen}
K_{\frac{g}{2}-2, 1}\bigl(C, K_C\otimes \eta\bigr)=0,
\end{equation}
whereas in odd genus
\begin{equation}\label{prymgreenodd}
K_{\frac{g-3}{2}, 1}\bigl(C, K_C\otimes \eta\bigr)=0 \ \ \mbox{ and }\ \ K_{\frac{g-7}{2}, 2}(C, K_C\otimes \eta)=0.
\end{equation}
}
The name of the conjecture is justified by the analogy with Green's Conjecture. In Section 4, we verify the Prym-Green Conjecture for all values of $g\leq 18$ and small $\ell$, with the exception of $2$-torsion in genus $g=8$ and $g=16$. Our findings suggest that very likely, for $2$-torsion and $g$ a multiple of $8$ (or perhaps a power of $2$), the Prym-Green Conjecture might actually be false. (With our methods, an experiment with $g=24$ would take about 6500 years.) Our verification of Conjecture A is computational via the use of \emph{Macaulay2}. We verify condition (\ref{prymgreen}) for $g$-nodal rational curves over a finite field (see Section 4 for details).  For $g:=2i+6$, we denote by
$$\cZ_{g, \ell}:=\Bigl\{[C, \eta]\in \cR_{g, \ell}: K_{i+1, 1}(C, K_C\otimes \eta)\neq 0\Bigr\},$$
the failure locus of Conjecture A. This is a virtual divisor on $\cR_{g, \ell}$.

\vskip 2pt

The second conjecture we address concerns the resolution of torsion bundles.

\noindent
{\bf B. Torsion Bundle Conjecture.} \emph{Let $[C,\eta]\in \cR_{g, \ell}$ be a general level $\ell$ curve of \emph{even} genus $g\geq 4$.
For each $1\leq k\leq \ell-2$, the module $\Gamma_C(\eta^{\otimes k}, L)$ has natural resolution, unless $$\eta^{\otimes (2k+1)}=\OO_C \  \mbox{ and }g \equiv 2\ \  \mathrm{mod}\ 4 \mbox{ and } \ {g-3 \choose {\frac{g}{2}-1}} \equiv 1 \ \mathrm{mod}\ 2.$$ Equivalently, one has the vanishing statement
\begin{equation}\label{conjectB}
K_{\frac{g}{2}-1, 1}(C; \eta^{\otimes k}, K_C\otimes \eta)=0.
\end{equation} In the exceptional cases there is precisely one extra syzygy}.

\vskip 3pt
Theorem \ref{general bundles} can be viewed as a weak form of Conjecture B. The first exceptional genera in the Conjecture B are $g=6,10,18,34$ or $66$. For levels $\ell\geq 3$, we set $g:=2i+2$ and denote the corresponding virtual divisor by
$$\mathcal{D}_{g, \ell}:=\Bigl\{[C, \eta]\in \cR_{g, \ell}: K_{i, 1}(C; \eta^{\otimes (\ell-2)}, K_C\otimes \eta)\neq 0\Bigr\}.$$
Conjecture B can be reformulated in the spirit of the \emph{Minimal Resolution Conjecture} \cite{FMP} for points on paracanonical curves as follows.  We view a divisor $Z\in |K_C\otimes \eta^{\otimes (1-k)}|$  as a $0$-dimensional subscheme of $\phi_L(C)\subset \PP^{g-2}$. Then using \cite{FMP} Proposition 1.6, Conjecture B is equivalent to the statement $b_{\frac{g}{2}, 1}(Z)=b_{\frac{g}{2}-1, 2}(Z)=0$, which amounts to the minimality of the number of syzygies of $Z\subset \PP^{g-2}$.

The exceptions to Conjecture B can be explained by (surprising) symmetries in the Koszul differentials (see Section 4). We verify Conjecture B for genus $g\leq 16$ and small level. In any of these cases $\mathcal{D}_{g, \ell}$ is a divisor on $\cR_{g, \ell}$.

\vskip 3pt
A prediction about the naturality of the resolution of the module $\Gamma_C(\eta, K_C)$ can also be made. This time we expect no exceptions and this turns out to be the case:

\begin{theorem}\label{thmC}
Let $[C,\eta]\in \cR_{g, \ell}$ be a general level $\ell$ curve of genus $g\ge 3$.
Then the resolution of $\Gamma_C(\eta, K_C)$ is natural (respectively pure for odd $g$), that is,
\begin{equation}\label{mrc}
K_{\lfloor\frac{g-1}{2}\rfloor, 1}\bigl(C; \eta, K_C\bigr)=0.
\end{equation}
\end{theorem}

The proof of  Theorem \ref{thmC} is by specialization to hyperelliptic curves. Via the following equality of cycles in the Jacobian of any curve of \emph{odd} genus, see \cite{FMP},
$$
C_{\frac{g-1}{2}}-C_{\frac{g-1}{2}}=\Bigl\{\xi\in \mbox{Pic}^0(C): K_{\frac{g-1}{2}, 1}(C; \xi, K_C)\neq 0\Bigr\},
$$
where the left hand side denotes the top difference variety of $C$,  Theorem \ref{thmC} admits the following geometrically transparent reformulation:

\begin{corollary}
For a general curve $C$ of odd genus, the top difference variety $C_{\frac{g-1}{2}}-C_{\frac{g-1}{2}}$ contains no non-trivial $\ell$-torsion points, for any $\ell\geq 2$.
\end{corollary}

To pass from Koszul vanishing statements to the birational structure of $\cR_{g, \ell}$, one calculates the cohomology classes of the universal failure loci of the conditions (\ref{prymgreen}), (\ref{conjectB}) and (\ref{mrc}) respectively on a suitable compact moduli of level $\ell$ curves.
The space $\cR_{g, \ell}$ admits a compactification $\rr_{g, \ell}$, which is the coarse moduli space associated to the smooth proper Deligne-Mumford stack $\ssrr_{g, \ell}$ of \emph{level twisted curves}, that is, triples $[\ssC, \eta, \phi]$, where $\ssC$ is a genus $g$ twisted curve, $\eta$ is a faithful line bundle on the stack $\ssC$ and $\phi:\eta^{\otimes \ell}\rightarrow \OO_{\ssC}$ is an isomorphism, see \cite{CF} and Section 1 of this paper. We denote by $f:\rr_{g, \ell}\rightarrow \mm_g$ the forgetful map.

An essential ingredient in the proof of Theorem \ref{koddim} is one of the main results of \cite{CF}. It is shown that the non log-canonical singularities of $\rr_{g, 3}$ do not impose adjunction conditions, that is, if $\epsilon: \widehat{\cR}_{g, 3}\rightarrow \rr_{g, 3}$ denotes a resolution of singularities, then for each $n\geq 1$, there exists an isomorphism at the level of spaces of global sections
$$\epsilon^*: H^0\Bigl(\rr_{g, 3}^{\mathrm{reg}}, K_{\rr_{g, 3}}^{\otimes n}\Bigr)\stackrel{\cong}\longrightarrow H^0\Bigl(\widehat{\cR}_{g, 3}, K_{\widehat{\cR}_{g, 3}}^{\otimes n}\Bigr).$$ A similar extension result has been proved in \cite{HM} for level one curves, that is, for the moduli space $\mm_g$ itself, and in \cite{FL} in the case of $\rr_{g, 2}$. We refer to \cite{CF} for partial generalizations of this extension result for higher levels. Using \cite{CF}, we thus conclude that the Kodaira dimension of $\rr_{g, 3}$ is equal to the Kodaira-Iitaka dimension of the canonical bundle $K_{\rr_{g, 3}}$. In order to prove that for a given $g$, the moduli space $\rr_{g, 3}$ is of general type, it suffices to express
$K_{\rr_{g, 3}}\in \Pic(\rr_{g, 3})$ as a positive combination of the Hodge class $\lambda\in \Pic(\rr_{g, 3})$, the class $[\ol{\mathfrak{D}}]$ of the closure of a certain effective divisor $\mathfrak{D}$ on $\cR_{g, 3}$ and boundary divisor classes corresponding to singular level curves. In our case, the divisor $\mathfrak{D}$ is a jumping locus for Koszul cohomology groups of paracanonically embedded curves, defines by one of the conditions (\ref{prymgreen}), (\ref{conjectB}) or (\ref{mrc}).

We denote by $\wt{\cM}_g$ the open
subvariety of $\mm_g$ classifying
irreducible stable curves of genus
$g$ and set $\wt{\cR}_{g, \ell}:=f^{-1}(\wt{\cM}_g)$.
The boundary $\wt{\cR}_{g, \ell}-f^{-1}(\cM_g)$ is the union of
three proper and closed subloci
$\Delta_0^{'}\cup \Delta_0^{''}\cup \Delta_{0}^{\mathrm{ram}}$,
where $\Delta_0^{''}$ and $\Delta_{0}^{\mathrm{ram}}$
can be characterized as follows (see Section 1, \S\ref{sect:boundary} for further details).
The locus $\Delta_{0}^{\mathrm{ram}}$ is
the ramification divisor of $f$.
The locus $\Delta_0^{''}$ is the locus of order-$\ell$ analogues of Wirtinger covers.
These boundary subloci are not irreducible in general and we refer to \cite[\S1.4.3-4]{CF} for a
decomposition into irreducible components.
When $\ell=2$ or $3$ however, the loci $\Delta_0^{'}$, $\Delta_0^{''}$ and
$\Delta_{0}^{\mathrm{ram}}$ are irreducible, whereas
for a prime level $\ell> 3$, both
$\Delta''_0$ and $\Delta_{0}^{\mathrm{ram}}$ are reducible and
decompose into $\lfloor \frac{\ell}{2} \rfloor$
irreducible components. For  $\Delta_{0}^{\mathrm{ram}}$, we have a decomposition
$\Delta_0^{\mathrm{ram}}=\bigcup_{a=1}^{\lfloor \ell/2\rfloor} \Delta_0^{(a)}$,
where the irreducible components are determined by a local index $a$ at the node of the generic level $\ell$ curve
(see Definitions \ref{defn:qslevel} and \ref{defn:twistlevel}) and yielding
$\QQ$-divisors $\delta_0^{(a)}=[\Delta_0^{(a)}]_{\QQ}$ at the level of the moduli stack,
fitting in the formula
$$f^*(\delta_0)=\delta_0'+\delta_0''+\ell \sum_{a=1}^{\lfloor \ell/2\rfloor} \delta_0^{(a)}.$$
Here $\delta_0$ is the $\QQ$-divisor attached to
$\Delta_0=\wt \cM_g- \cM_g$.
As shown in Section 1, \S\ref{sect:compositel} and \S\ref{sect:boundary},
the above identity extends word for word for possibly composite levels $\ell\ge 2$.

For any $\ell\ge 2$,
following the established practice of passing to
lower case symbols for the divisor classes on the moduli stack, we have the following results:

\begin{theorem}\label{mrc1}
Write $g=2i+1$ and $\ell\geq 2$. The class of the closure in $\wt{\cR}_{g, \ell}$ of the effective divisor
$\cU_{g, \ell}:=\bigl\{[C, \eta]\in \cR_{g, \ell}: K_{i, 1}(C; \eta, K_C)\neq 0\bigr\}$
is equal to
$$[\ol{\cU}_{g, \ell}]=\frac{1}{2i-1}{2i\choose i}\Bigl((3i+1)\lambda-\frac{i}{2}(\delta_0^{'}+\delta_0^{''})-\sum_{a=1}^{\lfloor \frac{\ell}{2}\rfloor}\frac{1}{2\ell} (i\ell^2+2a^2i-2a\ell i-a^2+a\ell)\ \delta_0^{(a)}\Bigr)\in \mathsf{Pic}(\wt{\mathsf{R}}_{g, \ell}).$$
\end{theorem}

Keeping $g=2i+1$ and setting $\ell=3$, in order to form an effective representative of $K_{\rr_{g, 3}}$, we use Theorem \ref{mrc1} together with the formula of the slope $s(\mm_{g, i+1}^1)=\frac{6(i+2)}{i+1}$ of the class of the closure in $\mm_g$ of the Hurwitz divisor \cite{HM}, \cite{EH}
$$\cM_{g, i+1}^1:=\{[C]\in \cM_g: W^1_{i+1}(C)\neq \emptyset\},$$ in order to obtain that, for suitable rational constants $\alpha, \beta >0$, the $\mathbb Q$-divisor class
\begin{equation}\label{combinatie}
\alpha\cdot [\ol{\cU}_{g, 3}]+\beta\cdot [f^*(\mm_{g, i+1}^1)]=\frac{6(2i+3)}{i+1}\lambda-2(\delta_0^{'}+\delta_0^{''})-4\delta_0^{(1)}\in \mathsf{Pic}(\wt{\mathsf{R}}_{g, 3})
\end{equation}
is effective. Comparing this formula against that of the canonical class $K_{\rr_{g, 3}}$ (see Section 1), we note that whenever the following inequality
$$\frac{6(2i+3)}{i+1}<13\Leftrightarrow i>5,$$
holds, the canonical class $K_{\rr_{g, 3}}$ is big. Using the extension result of \cite{CF}, we conclude that $\rr_{g, 3}$ is of general type for odd genus
$g\geq 13$. This argument also shows that when $g=11$ the class given in (\ref{combinatie}) is an effective representative for the canonical class; furthermore, one has the following inequalities
$$\kappa(\rr_{11, 3}, K_{\rr_{11,3}})\geq \kappa\bigl(\rr_{11,3}, f^*(\mm_{11,6}^1)\bigr)\geq \kappa(\mm_{11}, \mm_{11, 6}^1)=19,$$
where the last equality has been proved in \cite{FP}. It is an interesting open question whether the equality $\kappa(\rr_{11, 3})=19$ holds. It is known \cite{FV1} that
 the universal Picard variety over the moduli space $\mm_{11}$ has Kodaira dimension equal to $19$ as well.

\vskip 3pt
We compute the (virtual) class of the failure loci given by Conjectures A and B.
\begin{theorem}\label{prymgreenclass}
Set $g:=2i+6$ with $i\geq 0$ and $\ell\geq 2$. The virtual class of the closure in $\wt{\cR}_{g,\ell}$ of the locus $\cZ_{g, \ell}$ of level $\ell$ curves $[C, \eta]\in \cR_{g, \ell}$ with $K_{i+1, 1}(C, K_C\otimes \eta)\neq 0$ is equal to
$$[\ol{\cZ}_{g, \ell}]^{\mathrm{virt}}={2i+2\choose
i}\Bigl(\frac{3(2i+7)}{i+3}\lambda-(\delta_0'+\delta_0^{''})-\sum_{a=1}^{\lfloor \frac{\ell}{2}\rfloor} \frac{a^2-a\ell+\ell^2}{2}\delta_0^{(a)}\Bigr)\in \sPic(\wt{\ssR}_{g,\ell}).$$
\end{theorem}
We explain the meaning of this result. In Section 3 we construct tautological vector bundles $\cA$ and $\cB$ over the stack $\wt{\ssR}_{g, \ell}$ with $\mbox{rk}(\cA)=\mbox{rk}(\cB)$, as well as a vector bundle morphism $\varphi:\cA\rightarrow \cB$, such that over the open part $\cR_{g, \ell}\subset \rr_{g, \ell}$, the degeneracy locus of $\varphi$ equals the scheme $\cZ_{g, \ell}$. Accordingly, we define $[\zz_{g, \ell}]^{\mathrm{virt}}:=c_1(\cB-\cA)$. Whenever
the Prym-Green Conjecture holds, that is, $\varphi$ is generically non-degenerate and $\zz_{g, \ell}$ is a divisor, we have that $[\zz_{g, \ell}]^{\mathrm{virt}}-[\zz_{g, \ell}]$ is a (possibly empty) effective class entirely supported on the boundary of $\wt{\cR}_{g, \ell}$. In particular, the class computed in Theorem \ref{prymgreenclass} is effective. Next we describe the universal failure locus of Conjecture B.

\begin{theorem}\label{conjB}
Set $g:=2i+2\geq 4$ and $\ell\geq 3$ such that $i\equiv 1 \ \mathrm{mod}\ 2$ or ${2i-1\choose i}\equiv 0 \ \mathrm{mod}\ 2$.  The virtual class of the closure in $\wt{\cR}_{g, \ell}$ of the locus $\mathcal{D}_{g, \ell}$ of level $\ell$ curves $[C, \eta]\in \cR_{g, \ell}$ such that $K_{i, 1}(C; \eta^{\otimes ({\ell-2})}, K_C\otimes \eta)\neq 0$ is equal to
$$[\ol{\cD}_{g, \ell}]^{\mathrm{virt}}=\frac{1}{i-1}{2i-2\choose i}\Bigl((6i+1)\lambda-i(\delta_0^{'}+\delta_0^{''})-\frac{1}{\ell}\sum_{a=1}^{\lfloor \frac{\ell}{2}\rfloor} (i\ell^2+5a^2i-5ai\ell-2a^2+2a\ell)\delta_0^{(a)}\Bigr).$$
\end{theorem}

To establish Theorem \ref{koddim} in even genus, when $g\geq 14$ but $g\neq 16$, we can use the class $[\ol{\cZ}_{g, 3}]$ to show that $K_{\rr_{g, 3}}$ is big. In genus $g=16$, when the Prym-Green Conjecture appears to fail, we use the class $[\ol{\cD}_{g, 3}]$ instead. Interestingly, for $g=12$, the Prym-Green divisor $\ol{\cZ}_{12, 3}\equiv 13\lambda-2(\delta_0^{'}+\delta_0^{''})-\frac{14}{3}\delta_0^{(1)}$ has slope equal to that of the canonical divisor $K_{\rr_{12, 3}}$. However, one can form the effective linear combination
$$\frac{31}{36}[\ol{\cZ}_{12, 3}]+\frac{1}{36\cdot 7}[\ol{\cD}_{12, 3}]= \bigl(13-\frac{1}{12}\bigr)\lambda-2(\delta_0^{'}+\delta_0^{''})-4\delta_{0}^{(1)}=K_{\wt{\ssR}_{12, 3}}-\frac{1}{12}\lambda\in \mathrm{Eff}(\wt{\ssR}_{12, 3}),$$
thus showing that $\rr_{12, 3}$ is of general type as well.

\vskip 3pt
Section 5 is devoted to the surprising failure of the Prym-Green Conjecture for $g=8$ and $\ell=2$. We give a "probabilistic proof" of the fact that
for a general genus $8$ Prym canonical curve  $\phi_{K_C\otimes \eta}:C\hookrightarrow \PP^6$, the multiplication map
$$I_2(C, K_C\otimes \eta)\otimes H^0(C, K_C\otimes \eta)\rightarrow I_3(C, K_C\otimes \eta),$$
has a non-trivial $1$-dimensional kernel, corresponding to a syzygy of rank $6$.

A curve $\phi_L:C\hookrightarrow \PP^6$ corresponding to a general element $[C, L]$ of the universal Jacobian variety $\mathfrak{Pic}^{14}_8\rightarrow \cM_8$ is linked via five quadrics to a genus $14$ curve with general moduli $C'\subset \PP^6$, embedded such that $K_{C'}(-1)\in W^1_8(C')$ is a pencil of minimal degree.

The universal Koszul locus $\mathfrak{Kosz}:=\bigl\{[C, L]\in \mathfrak{Pic}_8^{14}: K_{1, 2}(C, L)\neq 0\bigr\}$ is a divisor that has at least two components
$\mathfrak{Kosz}_{6}$ and $\mathfrak{Kosz}_{7}$, distinguished by whether the
 extra syzygy has rank $6$ or $7$. The generic  curve in $\mathfrak{Kosz}_{6}$
corresponds via linkage to a curve in the Petri divisor $\mathcal{GP}_{14, 8}^1$ on $\cM_{14}$, while
the generic curve in $\mathfrak{Kosz}_{7}$ links to a 7-gonal curve of genus
14 such that $K_{C'}(-1)$ has a base point. Using a structure theorem for rank 6 syzygies in $\PP^{6}$, our findings show that
$\cR_{8,2}$ lies in the component $\mathfrak{Kosz}_{6}$, that is, the extra syzygy of a Prym-canonical curve is never of maximal rank.

\vskip 3pt
\noindent {\bf{Structure of the paper.}} The first section is dedicated to the geometry of the stack of twisted level $\ell$ curves. In Section 2 we describe the syzygy formalism needed to formulate Conjectures A and B and then we prove Theorem \ref{thmC}. In Section 3 proofs of Theorems \ref{mrc1}, \ref{prymgreenclass} and \ref{conjB} are provided, whereas in Section 4 we explain how to verify with \emph{Macaulay} the syzygy conjectures formulated in the Introduction. The last section is devoted entirely to genus $8$ and we explain the unexpected failure of the Prym-Green Conjecture on $\cR_{8, 2}$.

\noindent {\bf{Disclaimer.}} This paper is a collaborative effort marrying techniques ranging from computer algebra to stacks. The work of the second author is mainly reflected in sections 4 and 5, while the work of the first author is mainly reflected in sections 1--3.
\

\noindent {\bf{Acknowledgments.}} We are grateful to Anand Patel for pointing out an error in an earlier version of this paper, to Burcin Er\"ocal  and Florian Geiss for help
with the computational aspects, and to
Hans-Christian Graf von Bothmer and Alessandro Verra, who made suggestions that allowed us to improve the results of Section 5.
%
%

\section{Level $\ell$ curves}
This section contains background material concerning the moduli space $\rr_{g, \ell}$ and complements the paper \cite{CF}.
We begin by setting terminology. For a Deligne-Mumford stack $\mathsf{X}$ be denote by $X$ its associated coarse moduli space. The morphism
$\pi_{\mathsf X}:\mathsf{X}\rightarrow X$ is universal with respect to morphisms from $\mathsf{X}$ to algebraic spaces. The \emph{coarsening}
of a morphism $\mathsf{f}:\mathsf{X}\rightarrow \mathsf{Y}$ between Deligne-Mumford stacks is the induced morphism $f:X\rightarrow Y$ between
coarse moduli spaces. For a Deligne-Mumford stack $\mathsf X$ we denote by
$\sPic(\mathsf X)$
the Picard group of the stack with rational coefficients.


We fix two integers $g$ and $\ell \ge 2$, the genus
and the level.
Throughout \S\ref{sect:moduli} we
assume that $\ell$  is \emph{prime},
but we generalize all statements to any, possibly composite, level
$\ell\ge 2$ in \S\ref{sect:compositel}.

\subsection{The geometric points of the moduli space of level $\ell$ curves}\label{sect:point}
The geometric points of the moduli space
of level $\ell$ curves can be interpreted
in two relatively simple ways, as
quasi-stable $\ell$th roots and
twisted $\ell$th roots respectively. We are recall their definitions.

 A \emph{quasi-stable} curve is a nodal curve $X$ such that (i) for each smooth rational component $E\subset X$ the inequality $k_E:=|E\cap \ol{X-E}|\geq 2$ holds, and (ii) if $E, E'$ are rational components with $k_E=k_{E'}=2$, then $E\cap E'=\emptyset$. Rational components of $X$ meeting the rest of the curve in two points are called \emph{exceptional}. If $X$ is a quasi-stable curve, there exists a \emph{stabilization} morphism $\mathrm{st}:X\rightarrow C$, obtained by collapsing all rational curves $E\subset X$ with $k_E=2$. Abusing terminology, we say that $X$ is a \emph{blow-up} of the curve $C$.
\begin{definition}\label{defn:qslevel}
A \emph{quasi-stable level $\ell$ curve} consists of a triple $(X, \eta, \phi)$, where $X$ is a quasi-stable curve, $\eta\in \mbox{Pic}^0(X)$ is a locally free sheaf of total degree $0$ and $\phi:\eta^{\otimes \ell}\rightarrow \OO_X$ is a sheaf homomorphism satisfying the following properties:
\begin{enumerate}
\item $\eta_E=\OO_E(1)$, for every \emph{exceptional} component $E\subset X$;
\item $\phi$ is an isomorphism along each non-exceptional component of $X$;
\item if $E$ is an exceptional component and $\{p, q\}:=E\cap \ol{X-E}$, then $$\mbox{ord}_p(\phi)+\mbox{ord}_q(\phi)=\ell.$$
\end{enumerate}
\end{definition}
\noindent
The last condition refers to the vanishing orders of the section $0\neq \phi\in H^0(X, \eta^{\otimes (-\ell)})$.

We recall that a \emph{balanced twisted curve} is a Deligne-Mumford stack $\mathsf{C}$ whose coarse moduli space $C$ is a stable curve, and such that locally at a node, the stack $\mathsf{C}$ comes endowed with an action  of $\mathbb Z_{\ell}$ of determinant $1$ (this is equivalent
to impose the condition that $\mathsf{C}$ be smoothable).
\begin{definition}\label{defn:twistlevel}
A \emph{twisted $\ell$th root} $(\mathsf{C}\rightarrow T, \eta, \phi)$,
is a balanced twisted curve $\mathsf{C}$
of genus $g$ over a base $T$, a faithful line bundle $\mathsf{\eta}$ on $\mathsf{C}$ (i.e. a representable morphism
$\eta:\mathsf{C}\to \mathsf{B}\ZZ_\ell$)
and an isomorphism
$\phi:\eta^{\otimes \ell}\rightarrow \OO_{\mathsf{C}}$. If $\eta$ has order $\ell$ in $\sPic(\mathsf{C})$,
then $(\mathsf{C}\rightarrow T, \eta, \phi)$ is a \emph{level $\ell$ curve}.
\end{definition}

There exist
two $(3g-3)$-dimensional
projective schemes $\mathrm{Root}_{g,\ell}$ (see \cite{J,CCC}) and $\ol{\mathcal M}_g(\mathsf{B}\ZZ_\ell)$ (see \cite{AV, ACV})
whose (geometric) points represent isomorphism classes of quasi-stable $\ell$th roots and twisted $\ell$th roots, respectively. The moduli space $\ol{\mathcal R}_{g,\ell}$ of level
$\ell$ curves arises as a connected component of $\ol{\mathcal M}_g(\mathsf{B}\ZZ_\ell)$.
As we illustrate in detail below,
$\mathrm{Root}_{g,\ell}$  and $\ol{\mathcal M}_g(\mathsf{B}\ZZ_\ell)$
are not isomorphic unless $\ell=2$ or $3$, but
there exists a natural
morphism $\mathrm{nor}: \ol{\mathcal M}_g(\mathsf{B}\ZZ_\ell)\to \mathrm{Root}_{g,\ell}$,
which sets a bijection on the sets
of geometric points, and may be regarded,
scheme-theoretically, as a normalization morphism (for $\ell>3$ the singularities of $\mathrm{Root}_{g,\ell}$ are not normal).
The aim of \S\ref{sect:moduli} is to describe
the twisted $\ell$th root
$(\mathsf{C}\rightarrow T, \eta, \phi)$
and  the quasi-stable $\ell$th root
$(X, \eta, \phi)$
that correspond to each other under the map $\mathrm{nor}$.
In fact, following \cite{CGRR}, we
lift this correspondence to a correspondence between a
universal twisted $\ell$th root and a universal quasi-stable $\ell$th
root both defined on the moduli stack
of level $\ell$ curves. This correspondence
allows us to study the enumerative geometry
of the moduli of level curves both in scheme-theoretic and
stack-theoretic terms (see Remark \ref{rem:forkernelbundles}).

\subsection{The moduli stack of level $\ell$ curves (when $\ell$ is prime)}\label{sect:moduli}
We consider the categories of quasi-stable $\ell$th roots and of twisted $\ell$ roots. For the sake of clarity, in this section, we assume that $\ell$ is prime.

A family of quasi-stable $\ell$th roots consists of a triple $(f, \eta, \phi)$, where $f:X\rightarrow T$ is a flat family of quasi-stable curves, $\eta$ is a line bundle on $X$  and
$\phi:\eta^{\otimes \ell}\rightarrow \OO_{X}$ is a morphism of sheaves such that for each geometric point $t\in T$, the restriction $$\bigl(X_t:=f^{-1}(t), \ \eta_t:=\eta_{| X_t},\ \phi_{|X_t}:\eta_t^{\otimes \ell}\rightarrow \OO_{X_t}\bigr)$$ is a quasi-stable $\ell$th root as defined above. The category of level $\ell$ curves gives rise to a proper Deligne-Mumford stack $\mathsf{Root}_{g, \ell}$ with associated coarse moduli space $\mathrm{Root}_{g, \ell}$, see \cite{CCC, J}. Unfortunately, as already mentioned, this stack is singular as soon as $\ell>3$ (see \eqref{eq:nonnormal}).
 To obtain a smooth stack whose coarse moduli space is the normalization of $\mathrm{Root}_{g, \ell}$ one can use
{twisted $\ell$ roots} \cite{AV, ACV}. 

Indeed, the category $\ol{\pem}_g(\mathsf{B}\mathbb Z_{\ell})$
of twisted $\ell$th roots forms a smooth and proper
Deligne-Mumford stack, whose coarse moduli space is the normalization of $\mathrm{Root}_{g, \ell}$.
The connected component
parametrizing order-$\ell$ line bundles is the Deligne-Mumford
stack $\ol {\mathsf R}_{g,\ell}$;
the coarse space $\ol {\mathcal R}_{g,\ell}$ is the $(3g-3)$-dimensional projective variety studied in
this paper.
The following diagram is commutative
\begin{equation}\label{eq:nor}\xymatrix@R=6pt{
  \ol{\pem}_g(\mathsf{B}\mathbb Z_{\ell})\ar[dr]_{\mathsf{f}}  \ar[rr]^{\mathsf{nor}  }  & & \mathsf{Root}_{g, \ell} \ar[dl]^{\mathsf{h}}   \\
  & \ol{\pem}_g &      \\
                 }\end{equation}
and endows $\ol{\pem}_g(\mathsf{B}\mathbb Z_{\ell})$
with a universal twisted $\ell$th root and a universal quasi-stable $\ell$th root.
The universal twisted $\ell$th root
$$\Bigl(\mathsf{u}_{g,\ell}^{\mathsf C} : \mathsf{C}_{g,\ell}\to \ol{\pem}_g(\mathsf{B}\mathbb Z_{\ell}),\
\eta^{\mathsf C}_{g,\ell}\in \sPic(\mathsf{C}_{g,\ell}),\
\phi_{g,\ell}^{\mathsf C}:(\eta^{\mathsf C}_{g,\ell})^{\otimes \ell} \xrightarrow{\sim}
\mathcal O\Bigr)$$
consists of a universal balanced twisted curve,
a line bundle and an isomorphism.
The universal quasi-stable level $\ell$ curve pulled back via $\mathsf{nor}^*$
$$\Bigl({\mathsf{u}}_{g,\ell}^{X} : X_{g,\ell}\to \ol{\pem}_g(\mathsf{B}\mathbb Z_{\ell}),\
\eta^X_{g,\ell}\in \sPic(X_{g,\ell}),\
\phi_{g,\ell}^{X}:(\eta^{X}_{g,\ell})^{\otimes \ell} \to
\mathcal O\Bigl)$$
consists of a family
of quasi-stable curves,
endowed with a line bundle
and a homomorphism of line bundles.
For any morphisms $T\to \ol{\pem}_g(\mathsf{B}\mathbb Z_{\ell})$
from a scheme $T$,
we consider the pullbacks
$(X_{g,\ell})_T$  and
$(\mathsf{C}_{g,\ell})_T$; the stabilization of the quasi-stable
curve $(X_{g,\ell})_T$  and
the coarsening of the balanced twisted curve
$(\mathsf{C}_{g,\ell})_T$ coincide and yield
a representable
morphism
$C_{g,\ell}\to \ol{\pem}_g(\mathsf{B}\mathbb Z_{\ell})$
which may be also regarded as the universal stable curve pulled back from
$\ol{\mathsf{M}}_g$ via $\mathsf{f}$. On $C_{g,\ell}$, we have the following
identity of sheaves (see \cite{JPic} and, in particular, (1-3)
after Lemma 3.3.8, the proof of Theorem 3.3.9, and Figure 1;
see also \cite[Lem.~2.2.5]{CGRR} for a more complete statement).

\begin{proposition}\label{pro:univfamilies}
Over $\ol{\pem}_g(\mathsf{B}\mathbb Z_{\ell})$, consider the stabilization morphism
$\mathrm{st}_{g,\ell}: X_{g,\ell}\to C_{g,\ell}$
and the coarsening morphism
$\pi_{g,\ell} : \mathsf{C}_{g,\ell}\to C_{g,\ell}$.
On the universal stable curve $C_{g,\ell}$, we have the following identity between  coherent sheaves
$(\pi_{g,\ell})_*(\eta^{\mathsf C}_{g,\ell})=(\mathrm{st}_{g,\ell})_*(\eta^X_{g,\ell}).$\qed
\end{proposition}

Let us fix a closed
point $\ol \tau:=[\mathsf{C},\eta,\phi]$  representing a twisted
$\ell$th root and its image  $\ol t:=[X, \eta, \phi]$ representing a
quasi-stable $\ell$th root.
By the above proposition, the coarsening $C$ of $\mathsf{C}$ equals
the stabilization of $X$. Consider the exceptional components $E^1, \ldots, E^k\subset X$ and write $\{p_i, q_i\}:=E^i\cap \ol{X-E^i}$ and $a_i:=\mbox{ord}_{p_i}(\phi)$ and $b_i:=\mbox{ord}_{q_i}(\phi)$, for $i=1, \ldots, k$. Then $a_i+b_i=\ell$ and we have $\mbox{gcd}(a_i, b_i)=1$ (because $\ell$ is prime).
The only nontrivial stabilizers of $\mathsf{C}$ occur at
the nodes $\mathsf{n}_1$,\dots ,$\mathsf{n}_k$
of $\mathsf{C}$ mapping to the nodes
$\mathrm{st}(E_1), \ldots, \mathrm{st}(E_1)$ of $C$. The local picture
at $\mathsf{n}_i$ is the spectrum of
$\mathbb{C}[\wt x_i,\wt y_i]/(\wt x_i\wt y_i=0)$ and
the local coordinates $x_i$ and $y_i$
of $\ol{ X-E^i}$ at $p_i$ and at $q_i$
are related to $\wt x_i$ and to $\wt y_i$ by
$\wt x_i^\ell=x_i$ and $\wt y_i^{\ell}=y_i$.
The action of $\ZZ_\ell$ on
the local picture $\Spec R[t]\to \Spec R$
of $\eta\to \mathsf{C}$ at $\mathsf{n_i}$ is given by
$(\wt x_i,\wt y_i)\mapsto (\xi_\ell \wt x_i, \xi_\ell^{-1}\wt y_i)$ and $t_i\mapsto \xi_\ell^{a_i}t_i$.
Following \cite{CF}, we refer to the pair $(a_i, b_i)$ as the \emph{multiplicity} of the quasi-stable $\ell$-root $[X, \eta, \phi]$ along the exceptional component $E^i$ and, equivalently,
of the twisted $\ell$th root
at $\mathsf{n}_i$.

We further describe the local picture of the universal quasi-stable
$\ell$th root over $\mathsf{Root}_{g,\ell}$ along one exceptional component $E=E^i$ meeting the rest of $X$ at $p$ and $q$ and having multiplicity $(a, b):=(a_i, b_i)$. It is proved in \cite[\S3.1]{CCC},  that if one denotes by $C_{a, b}\subset \mathbb A^2_{w, z}$ the affine plane curve given by the equation $w^a=z^b$ and the following surface by
\begin{equation}\label{eq:nonnormal}
S_{a, b}:=\Bigl\{\bigl((x, y, z, w), [s_0:s_1]\bigr)\in \mathbb A_{x, y, z, w}^4\times \PP^1: xs_0=ws_1, \ ys_1=zs_0,\  w^a=z^b\Bigr\},  \end{equation}
then the covering $f_{a, b}:S_{a, b}\rightarrow C_{a, b}$ given
by $f_{a, b}:((x, y, w, z), [s_0:s_1])\mapsto (w, z)$ is a local model for the simultaneous smoothing of the nodes $p$ and $q$ within the
universal quasi-stable $\ell$th root over $\mathsf{Root}_{g, \ell}$. In particular, as soon as $a, b>1$, the space $C_{a, b}$ and hence the moduli space $\mathrm{Root}_{g, \ell}$ are not normal. The normalization $\wt{f}_{a, b}:\wt{S}_{a, b}\rightarrow \mathbb A_{\tau}^1$ of $f_{a, b}$ is constructed by setting
\begin{equation}\label{eq:norm}
\wt{S}_{a, b}:=\Bigl\{\bigl((x, y, \tau), [s_0:s_1]\bigr)\in \mathbb A^3_{x, y, \tau}\times \PP^1:xs_0=\tau^bs_1, ys_1=\tau^a s_0\Bigr\},                                                                                                                                           \end{equation}  and mapping
$\mathbb A_{\tau}^1\rightarrow C_{a, b}$ via $\tau\mapsto (\tau^b, \tau^a)$. At the level of the normalization, the point $q\in X$ corresponds to the $A_{a-1}$-singularity $\bigl((0, 0, 0), [1: 0]\bigr)\in \wt{S}_{a, b}$, whereas
the $A_{b-1}$-singularity $\bigl((0, 0, 0),[0:1]\bigr)\in \wt{S}_{a, b}$ corresponds to the point $p\in X$.
Globalizing this description,
one obtains a local picture of the diagram \eqref{eq:nor} at the points
$\ol \tau$ and $\ol t$.
There exist local coordinates such that
\begin{align}\label{eq:localrings}
\hat{\OO}_{\mathsf{Root}_{g, \ell}, \ \ol t}&=\mathbb C[[w_1, z_1, \ldots, w_k, z_k, t_{k+1}, \ldots, t_{3g-3}]]/(w_i^{a_i}=z_i^{b_i}), \mbox{ for } i=1, \ldots, k,\\
\hat{\OO}_{\ol{\pem}_g, \mathrm{st}(X)}&=\mathbb C[[t_1, \ldots, t_{3g-3}]],\\
\hat{\OO}_{\ol{\pem}_g(\mathsf{B}\mathbb Z_{\ell}), \ \ol \tau}&=\mathbb C[[\tau_1, \ldots, \tau_k, t_{k+1}, \ldots, t_{3g-3}]].
\end{align}
Locally, the morphisms $\mathsf{f}$, $\mathsf{nor}$ and $\mathsf{h}$ are given by
\begin{align}\label{eq:loccoordf}\mathsf{f}:t_i\mapsto \begin{cases}\tau_i^\ell &\text{$i\le k$}\\
                                     t_i &\text{$i>k$}
                                    \end{cases} \qquad \mathsf{nor}:\begin{cases}w_i\mapsto \tau_i^{b_i} &\text{$i\le k$}\\
z_i\mapsto \tau_i^{a_i}&\text{$i\le k$}\\
                                     t_i\mapsto t_i &\text{$i>k$}
                                    \end{cases}\qquad
\mathsf{h}:t_i\mapsto \begin{cases}w_iz_i &\text{$i\le k$}\\
                                     t_i &\text{$i>k$.}
                                    \end{cases}
\end{align}
We regard $X$ as a fiber of
a universal quasi-stable curve on
$\ol{\pem}_g(\mathsf{B}\mathbb Z_{\ell})$
over $\ol \tau$. Set $\tau:=\tau_i$, $(a,b):=(a_i,b_i)$.
Then, at $p$, we obtain an $A_{b-1}$-singularity
of equation $x( {s_0}/{s_1})=\tau^{b}$
along the locus where the node $p$ persists.
At
$q$, we obtain an $A_{a-1}$-singularity
 of equation
$y({s_1}/{s_0})=\tau^{a}$ along the locus where
the node $q$ persists. In view of
intersection theory computations, following \cite{CGRR}, we
desingularize such singularities. For
any \'etale morphism  $T\to
\ol{\pem}_g(\mathsf{B}\mathbb Z_{\ell})$ the desingularization of
$(X_{g,\ell})_T$
yields a semi-stable curve, globally on $\ol{\pem}_g(\mathsf{B}\mathbb Z_{\ell})$,
$$X'_{g,\ell}\to X_{g,\ell} \to \ol{\pem}_g(\mathsf{B}\mathbb Z_{\ell})$$
equipped with a line bundle $\P:=\eta^{X'}_{g,\ell}$ and a homomorphism
$\Phi=\phi^{X'}_{g,\ell}$ from the $\ell$th tensor power of $\P$ to $\OO$:
the pull-backs of $\eta^X_{g,\ell}$ and $\phi^{X}_{g,\ell}$.
The curve
$X'$ over $\ol \tau\in \ol{\pem}_g(\mathsf{B}\mathbb Z_{\ell})$ is
obtained by iterated blow-ups; \emph{i.e.} for a fibre of $X_{g,\ell}$ of the form
$C'\cup E$ with $E$ exceptional and $E\cup C'=\{p,q\}$.
One inserts a
chain of $\ell-a-1$ rational curves $E_1, \ldots, E_{\ell-a-1}$ at $p$
and a  chain of $a-1$ rational curves $E_{\ell-a+1}, \ldots, E_{\ell-1}$ at $q$. Setting $E_{\ell-a}:=E$,
we obtain the nodal semi-stable curve
$C'\cup E_1\cup \ldots \cup E_{\ell-1}$ of Figure \ref{fig: X'}, where $C'=\ol{(X-E)}$.

By iterating this procedure at all exceptional curves we get $X'$, whose
stable model
is $C=\mathrm{st}(X)$, obtained by contracting the chains
of the form $E_1\cup\ldots\cup E_{\ell-1}$: this explains that
the singular locus  $C_{g,\ell}\to \ol{\pem}_g(\mathsf{B}\mathbb Z_{\ell})$
is formed by $A_{\ell-1}$-singularities (since
$C_{g,\ell}$ is the pullback via $\mathsf{f}$
of the universal stable curve
of $\ol {\mathsf{M}}_g$, this may be regarded as
a consequence of the local description of $\mathsf{f}$ of \eqref{eq:loccoordf}).
The restriction of
$\P=\eta^{X'}_{g,\ell}$ to $X'$ is a line bundle $\eta'$ satisfying
$\eta'_{E_i}=\OO_{E_i}$ for $i\neq \ell-a$ and $\eta'_{E_{\ell-a}}=\OO_{E_{\ell-a}}(1)$.

\begin{figure}[htb!]
\centering%
\includegraphics[width=5cm, height=3cm]{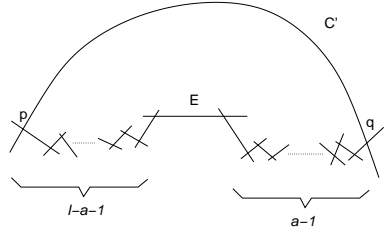}
\caption{The curve $X'$.}
\label{fig: X'}
\end{figure}


\begin{remark}\label{rem:forkernelbundles}
There is an isomorphism $H^0(X', \omega_{X'}\otimes \eta')\cong H^0(X, \omega_X\otimes \eta)$ and we often identify the two spaces. This may be regarded as a consequence of the
 fact that $X_{g,\ell}'\to X_{g,\ell}$  is crepant (the relative canonical bundles
match under pullback) and of Proposition \ref{pro:univfamilies}.
In practice this means that any cohomological question involving kernel bundles on a twisted $\ell$th root can be settled by working, equivalently,
over
$\ol{\pem}_g(\mathsf{B}\mathbb Z_{\ell})$, with the
twisted $\ell$th root $(\mathsf{C}_{g,\ell},\eta^{\mathsf{C}}_{g,\ell},\phi^{\mathsf{C}}_{g,\ell})$, with
the quasi-stable level
$\ell$ curve $(X_{g,\ell},\eta^X_{g,\ell},\phi^X_{g,\ell})$,
or with the pullback $\P:=\eta^{X'}_{g,\ell}$ of $\eta^X_{g,\ell}$ on the
semi-stable curve $X'_{g,\ell}$.
\end{remark}

\subsection{The moduli stack of level $\ell$ curves (for any level $\ell$)}\label{sect:compositel}
Let us assume only $\ell\in \ZZ_{\ge 2}$, without any
condition on $\ell$ being prime.
The local picture of $\mathsf{Root}_{g,\ell}$
\eqref{eq:localrings} is still valid; in particular the
local model for the curve smoothing the nodes $p$ and $q$
of an exceptional component $E$ is still $S_{a,b}$ where $(a,b)$ are
the multiplicites of the quasistable
$\ell$th root at the nodes. However, since $a$ and $b$ are no longer
necessarily coprime,
$S_{a,b}$ has $d:= \mbox{gcd}(a, b)$
irreducible components, each one isomorphic
to $S_{\frac{a}{d}, \frac{b}{d}}$. Then the local picture of $\mathsf{f}$,
$\mathsf{nor}$ and $\mathsf{h}$ of \eqref{eq:nor} is
\begin{align}\label{eq:loccoordf2}\mathsf{f}:t_i\mapsto \begin{cases}\tau_i^{\ell/d_i} &\text{$i\le m$}\\
                                     t_i &\text{$i>m$}
                                    \end{cases}
\qquad \mathsf{\wt{nor}}:\begin{cases}w_i\mapsto \tau_i^{\ell-a_i/d_i}
&\text{$i\le k$}\\
z_i\mapsto \tau_i^{a_i/d_i}&\text{$i\le k$}\\
                                     t_i\mapsto t_i &\text{$i>k$}
                                    \end{cases}\qquad
\mathsf{h}:t_i\mapsto \begin{cases}w_iz_i &\text{$i\le k$}\\
                                     t_i &\text{$i>k$,}
                                    \end{cases}
\end{align}
for $d_i=\mbox{gcd}(a_i,b_i)$.
This happens because the stabilizers of a node of $\mathsf{C}$
attached to an exceptional component whose multiplicities
satisfy $d_i=\mbox{gcd}(a_i,b_i)$ have order $\ell/d_i$.
Notice that the normalization
morphism $\mathsf {nor}$ induces a surjection at the level
of closed points, but does not induce an injection
as soon as $\ell$ is composite. Furthermore the desingularization
of the pullback of the
universal quasi-stable curve on $\ol{\mathsf{M}}_{g}(\mathsf B\ZZ_\ell)$
is a semistable curve with chain of rational curves of length
$\ell/d_i-1$ over the node corresponding to the the
exceptional component $E^i$. Fortunately, there is a simple way to avoid these exceptions.

In \cite{C} the first author introduced a smooth modification $\wt{\pem}_g(\mathsf{B}\mathbb Z_{\ell})$ of $\ol{\pem}_g(\mathsf{B}\mathbb Z_{\ell})$, which allows to reproduce,
even when $\ell$ is composite, the same configuration
with chains of $\ell-1$ rational curves illustrated in Figure \ref{fig: X'}.
The stack $\wt{\pem}_g(\mathsf{B}\mathbb Z_{\ell})$
parametrizes
triples $(\mathsf{C}, \eta, \phi)$ where $\mathsf{C}$ is
a twisted balanced curve with stabilizer of order $\ell$ at \emph{every} node, $\eta$ is a line bundle on $\mathsf{C}$, and
$\phi$ is an isomorphism identifying $\eta^{\otimes \ell}$ with $\OO$.
We do not impose any condition on faithfulness of the line bundle;
i.e. $(\mathsf{C}, \eta, \phi)$ may be
a non-representable map $\mathsf{C}\to \mathsf B\ZZ_\ell$.
In this way, by relaxing the condition on faithfulness and by imposing stabilizers of order $\ell$, the local picture
of $\mathsf{C}$ at each node of the singular locus
$\mathrm{Sing}({\mathsf{C}})=
\{\mathsf{n}_1, \ldots, \mathsf{n}_m\}$ is given by the spectrum of
$R=\mathbb C [\wt x_i,\wt y_i]/(\wt x_i\wt y_i=0)$ with
$\ZZ_\ell$ operating as
$(\wt x_i,\wt y_i)\mapsto (\xi_\ell \wt x_i, \xi_\ell^{-1}\wt y_i)$.
The line bundle $\eta\to \mathsf{C}$ has
local picture
$\Spec R[t_i]\to \Spec R$
with $\ZZ_\ell$ operating as
$t_i\mapsto \xi_\ell^{a_i}t_i$ with $a_i\in \{0, \ldots, \ell-1\}$,
not necessarily prime to $\ell$, for $i=1,\dots, m$.
As a consequence, there exist local coordinates at the closed
point $\wt \tau\in \wt{\pem}_g(\mathsf{B}\mathbb Z_{\ell})$ representing $(\mathsf{C},\eta,\phi)$
\begin{align*}
\hat{\OO}_{\wt{\pem}_g(\mathsf{B}\mathbb Z_{\ell}), \ \wt \tau}&=\mathbb C[[\wt\tau_1, \ldots, \wt\tau_m, t_{m+1}, \ldots, t_{3g-3}]].
\end{align*}
and a
morphism $\mathsf{e}:\wt{\pem}_g(\mathsf{B}\mathbb Z_{\ell})\to \ol{\pem}_g(\mathsf{B}\mathbb Z_{\ell})$ fitting in
\begin{equation}\label{eq:nor2}\xymatrix{
  \wt{\pem}_g(\mathsf{B}\mathbb Z_{\ell})\ar[r]^{\mathsf{e}}\ar[rrd]_{\wt{\mathsf f}}&
\ol{\pem}_g(\mathsf{B}\mathbb Z_{\ell})\ar[dr]^{\mathsf{f}}  \ar[rr]^{{\mathsf{nor}}  }  & & \mathsf{Root}_{g, \ell} \ar[dl]^{\mathsf{h}}   \\
&  & \ol{\pem}_g &      \\
                 }\end{equation}
given (at the level of
local rings) by $\mathsf e: \tau_i\mapsto \wt \tau_i^{d_i}$ for $i\le m$  and by
$\wt{\mathsf f}: t_i\mapsto \wt {\tau}_i^\ell$ for $i\le m$.

The morphism $\mathsf e$ yields an isomorphism at the level of coarse spaces. Therefore, as in \cite{CGRR}, when $\ell$ is not prime,
we can work throughout the rest of the paper with the substack
$\ol{\mathsf R}_{g,\ell}$ arising as the connected component
of $\wt{\pem}_g(\mathsf{B}\mathbb Z_{\ell})$ of
triples $(\mathsf C,\eta,\phi)$ where $\eta$ has order $\ell$.
This is a smooth Deligne-Mumford stack whose coarse space
is a projective $(3g-3)$-dimensional variety $\ol{\mathcal  R}_{g,\ell}$.
With a slight abuse of notation, let us denote by $X_{g,\ell}$ the
pullback of the universal quasi-stable curve on
$\ol{\mathsf R}_{g,\ell}$ via $\mathsf {nor}\circ \mathsf e$.
By construction, the singularities arising at the nodes $p_i$ and $q_i$
of an exceptional curve $E_i$
are of type $A_{\ell-a_i-1}$ and $A_{a_i-1}$ and the
desingularization $X'_{g,\ell}$ is fibred in semi-stable curve
obtained by inserting chains of $\ell-a_i-1$ rational curves at $p_i$
and chains of $a_i-1$ rational curves at $q_i$ as in Figure \ref{fig: X'}.
Remark \ref{rem:forkernelbundles} generalizes word for word. From
now on we shall work with
$\ol{\mathsf R}_{g,\ell}\subset \wt{\pem}_g(\mathsf{B}\mathbb Z_{\ell})$ equipped with the universal
semistable curve
$X'_{g,\ell}$ and with the universal line bundle $\P:=\eta^{X'}_{g,\ell}$
pulled back from $\mathsf{Root}_{g,\ell}$.

\subsection{The boundary divisors of $\rr_{g, \ell}$.}\label{sect:boundary}
We briefly discuss the geometry of the boundary divisors of $\rr_{g, \ell}$ by describing the structure of the fibre $f^{-1}([C])$ corresponding to a general point of each boundary of the boundary divisors $\Delta_0, \ldots, \Delta_{\lfloor \frac{g}{2}\rfloor}$ of $\mm_g$. Let us assume first that $C:=C_1\cup_p C_2$ is a transverse union of two smooth curves $C_1$ and $C_2$ of genus $i$ and $g-i$ respectively. If $[X, \eta, \phi]\in f^{-1}([C])$, then necessarily $X=C$ and $\eta$ is uniquely determined by the data of two line bundles $\eta_{C_1}\in \mbox{Pic}^0(C_1)[\ell]$ and $\eta_{C_2}\in \mbox{Pic}^0(C_2)[\ell]$. Depending on which of these line bundles is trivial, we define the boundary divisors $\Delta_i, \Delta_{g-i}$ and $\Delta_{i:g-i}$ respectively. For $1\leq i\leq g-1$, the general point of $\Delta_i$ corresponds to a level curve of compact type
$$\bigl[C_1\cup_p C_2, \eta_{C_1} \text{ of order $\ell$},\ \eta_{C_2}\cong\OO_{C_2}\bigr]\in \rr_{g, \ell}.$$
Finally, we denote by $\Delta_{i:g-i}$ the closure in $\rr_{g, \ell}$ of the locus of twisted level curves on $C_1\cup C_2$ such that $\eta_{C_1}\not \cong \OO_{C_1}$ and $\eta_{C_2}\not\cong \OO_{C_2}$. Denoting by $\delta_i:=[\Delta_i]_{\mathbb Q}, \delta_{g-i}:=[\Delta_{g-i}]_{\mathbb Q}$, $\delta_{i:g-i}:=[\Delta_{i:g-i}]_{\mathbb Q}$ the corresponding classes in $\Pic(\ssrr_{g, \ell})$, we have the following relation, showing that the morphism of stacks $\mathsf{f}:\ol{\mathsf{R}}_{g, \ell}\rightarrow \rem_g$ is \'etale over $\Delta_i$ where $i\geq 1$.
\begin{equation}\label{di}
\mathsf{f}^*(\delta_i)=\delta_i+\delta_{g-i}+\delta_{i:g-i}.
\end{equation}
For $i=1$ and $\ell\geq 3$, observe that at the level of coarse moduli spaces the formula $f^*(\Delta_1)=2(\Delta_1+\Delta_{1:g-1})+\Delta_{g-1}$ holds, in particular $f$ is ramified along $\Delta_1$ and $\Delta_{1:g-1}$.

Suppose now that $[C]\in \Delta_0$ is a general irreducible $1$-nodal curve of genus $g$ with normalization $\mbox{nor}:C'\rightarrow C$, and let $p, q\in C'$ be such that $\mbox{nor}(p)=\mbox{nor}(q)\in \mbox{Sing}(C)$. Assume that $[X, \eta, \phi]\in \pi^{-1}([C])$. We write down the exact sequence
\begin{equation}\label{nor}
1\longrightarrow \mathbb Z_{\ell}\longrightarrow \mbox{Pic}^0(C)[\ell]\stackrel{\mathrm{nor}^*}\longrightarrow \mbox{Pic}^0(C')[\ell]\longrightarrow 0.
\end{equation}
If $X=C$, we set $\eta_{C'}:=\mathrm{nor}^*(\eta)\in \mbox{Pic}^0(C')[\ell]$. We denote by $\Delta_0^{'}$ the closure in $\rr_{g, \ell}$ of the locus of level curves $[C, \eta, \phi]$ as above, where $\eta_{C'}\not \cong \OO_{C'}$.
An $\ell$-torsion line bundle on $C$ is determined by the choice of $\eta_{C'}$ and the choice of a $\mathbb Z_{\ell}$-gluing of the fibres $\eta_{C'}(p)$ and $\eta_{C'}(q)$. We observe that, when $\ell$ is prime,
$\mbox{deg}(\Delta_0^{'}/\Delta_0)=\ell(\ell^{2g-2}-1)$ (see \cite{CF} for the general statement).

We denote by $\Delta_0^{''}$ the closure in $\rr_{g, \ell}$ of the locus of level curves $[C, \eta, \phi]$
such that $\eta_{C'}\cong\OO_{C'}$ (we referred to these curves as Wirtinger covers in the introduction; these arise
as the preimage of a constant section via the morphism $\eta \to \OO$ between total spaces).
Using (\ref{nor}), an order-$\ell$ line bundle $\eta\in \mbox{Pic}^0(C)[\ell]$ with $\mathrm{nor}^*(\eta)=\OO_{C'}$ is determined by a
$\mathbb Z_{\ell}^*$-gluing of the fibres $\eta_{C'}(p)$ and $\eta_{C'}(q)$, that is, a root of unity $\xi_{\ell}^a\in \mathbb Z_{\ell}^*$, with $1\leq a\leq \ell-1$  prime to $\ell$, such that sections $\sigma\in \eta_{C'}$ that descend to $C$ are characterized by the equation $\sigma(p)=\xi_{\ell}^a \sigma(q)$. Since by reversing the role of $p$ and $q$ we interchange $a$ and $\ell-a$, we obtain the a decomposition into irreducible $\lfloor l/2\rfloor$ components, each of them of order $2$ over $\Delta_0$.
For any $\ell\in \ZZ_{\ge 2}$ we get $\mbox{deg}(\Delta_0^{''}/\Delta_0)=\ell-1$.

\vskip 3pt
We now consider the case $X=C'\cup_{\{p, q\}} E$, where $E$ is an exceptional component. Then by definition $\eta_E=\OO_E(1)$, therefore $\mbox{deg}(\eta_{C'})=-1$. Furthermore, there exists an integer $1\leq a\leq \ell-1$ such that $\eta_{C'}^{\otimes (-\ell)}=\OO_{C'}(a\cdot p+(\ell-a)\cdot q)$. Let us denote by $\Delta_0^{(a)}$ the closure in $\rr_{g, \ell}$ of the locus of such points. By switching the role of $p$ and $q$ we can obviously restrict ourselves to the case $1\leq a\leq \lfloor \frac{\ell}{2}\rfloor$. Since the choice of $\eta_{C'}\in \mbox{Pic}^{-1}(C')$ as above uniquely determines the level curve $[C'\cup E, \eta, \phi]$, it follows that, when $\ell$ is prime,  $\mbox{deg}(\Delta_0^{(a)}/\Delta_0)=2\ell^{2g-2}$ (the factor $2$ accounts for the possibility of interchanging $p$ and $q$). We refer to \cite{CF} for the treatment of the general case $\ell\in \ZZ_{\ge 2}$ and the decomposition into irreducible components.

We follow again the convention of passing to
lower case symbols for the divisor classes in the moduli stack $\ol{\mathsf{R}}_{g,\ell}$:
 $\delta_0'=[\Delta_0']_{\QQ}$, $\delta_0''=[\Delta_0'']_{\QQ}$ and
$\delta_0^{(a)}=[\Delta_0^{(a)}]_{\QQ}$. As
a direct consequence of the local description \eqref{eq:loccoordf} the morphism $\mathsf{f}$ is
\'etale over  $\Delta_0'$ and $\Delta_0''$ and
ramified with order $\ell$
at $\Delta_0^{(a)}$.
(When $\ell$ is not prime the order of ramification at $\Delta_0^{(a)}$ is still $\ell$ because
$\ol{\mathsf{R}}_{g,\ell}$ as been defined as a substack of the
covering $\wt{\mathsf{M}}_g(\mathsf{B}\ZZ_\ell)$ of $\ol{\mathsf{M}}_g(\mathsf{B}\ZZ_\ell)$). Summarizing
we have the following relation in $\Pic(\ssrr_{g, \ell})$
$$\mathsf{f}^*(\delta_0)=\delta_0^{'}+\delta_0^{''}+\ell\sum_{a=1}^{\lfloor \frac{\ell}{2}\rfloor} \delta_0^{(a)}.$$
%

\begin{proposition}
For $\ell\geq 3$, the canonical class of the coarse moduli space $\rr_{g, \ell}$ is equal to
$$K_{\rr_{g, \ell}}=13\lambda-2(\delta_0^{'}+\delta_0^{''})-(\ell+1)\sum_{a=1}^{\lfloor \frac{\ell}{2} \rfloor} \delta_0^{(a)}-2\sum_{i=1}^{\lfloor \frac{g}{2} \rfloor} (\delta_i+\delta_{g-i}+\delta_{i:g-i})-\delta_{g-1}\in \mathsf{Pic}(\ssrr_{g, \ell}).$$
\end{proposition}
\begin{proof}
It follows from the Hurwitz formula $K_{\rr_{g, \ell}}=f^*K_{\mm_g}+\delta_1+\delta_{1:g-1}+(\ell-1)\sum_{a=1}^{\lfloor \frac{\ell}{2}\rfloor} \delta_0^{(a)}$, coupled with the expression for $K_{\mm_g}$, cf. \cite{HM} Theorem 2 and formula (\ref{di}).
\end{proof}

\subsection{The geometry of the universal semi-stable level $\ell$ curve} \label{sect:universal}
Consider the universal semi-stable curve ${{{\mathsf u}}}:X'_{g,\ell}\rightarrow \ssrr_{g, \ell}$
introduced in \S\ref{sect:moduli},
endowed with the tautological bundle $\P:= \eta_{g,\ell}^{X'}$
and the homomorphism of line bundles $$\Phi:\P^{\otimes \ell}\rightarrow \OO_{X'_{g,\ell}}.$$
For sake of clarity and in preparation of
the enumerative geometry studied in this paper, we detail the
proof of a preliminary result on the self-intersection of
$c_1(\P)$.
First, let us fix an integer $1\leq a\leq \lfloor \frac{\ell}{2}\rfloor$ and observe that the the pull-back of the boundary divisor $\Delta_0^{(a)}$ in $\ol{\mathsf{R}}_{g,\ell}$ splits into irreducible components
$${{{\mathsf u}}}^*(\Delta_0^{(a)})=\E_1^{(a)}+\ldots+\E_{{\ell}-1}^{(a)},$$
where $\E_i^{(a)}$ parametrizes the closure of the locus of
points that lie on the $i$th rational component of the chain
$E_1\cup\dots\cup E_{\ell-1}$ of $\ell-1$ rational lines
over a general point of $\Delta_0^{(a)}$.
%
\begin{proposition}\label{features} The following relations hold in $\mathsf{Pic}(\ssrr_{g, \ell})$:
\newline
\noindent\ (1) ${{{\mathsf u}}}_*\bigl([\E_i^{(a)}]\cdot [\E_{i+1}^{(a)}]\bigr)=\delta_0^{(a)}$ and ${{{\mathsf u}}}_*\bigl([\E_i^{(a)}]^2)=-2\delta_0^{(a)}$, for each $1\leq i\leq \ell-1$.
\newline
\noindent\ (2) ${{{\mathsf u}}}_*\bigl(c_1(\P)\cdot c_1(\omega_{\mathsf u})\bigr)=0$.
\newline
\noindent\ (3) ${{{\mathsf u}}}_*(c_1^2(\P))=-\sum_{1\le a\le \lfloor \frac{\ell}{2}\rfloor} \frac{a(\ell-a)}{\ell}\delta_0^{(a)}$.
\end{proposition}
\begin{proof}
The first two statements being immediate, we proceed to the last one (see also \cite[Lem.~3.1.4]{CGRR}).
 For each $1\leq a\leq \lfloor \frac{\ell}{2}\rfloor $,
we fix a general quasi-stable $\ell$th root $(X^{a}, \eta, \phi)$ over $\Delta_{0}^{(a)}$.
Here $X^{a}:=C'\cup_{\{p, q\}} E$ is a quasi-stable
curve and $\eta_{C'}^{\otimes \ell}=\OO_{C'}(-a\cdot p-(\ell-a)\cdot q)$. The corresponding semi-stable fibre
of $X'_{g,\ell}$ mapping to $X^a$ is a semi-stable curve $X'$ obtained by setting
$E_{\ell-a}^{(a)}:=E$, and gluing to the points $p, q\in C$
smooth rational curves $E_1^{(a)}, \ldots, E_{\ell-a-1}^{(a)}$,
$E_{\ell-a}^{(a)}$, $E_{\ell-a+1}^{(a)}, \ldots, E_{\ell-1}^{(a)}$ forming a chain.
The line bundle $\eta'$ on $X'$ is the restriction $\P_{|X'}$
isomorphic to $\eta_{C'}$ on $C'$, to $\OO$ on ${E_i^{(a)}}$
for $i\neq \ell-a$, and to
$\OO(1)$ on ${E_{\ell-a}^{(a)}}$.
By construction we have that $X'\cap \E_i^{(a)}=E_i^{(a)}.$

We claim that the vanishing locus of the section $\Phi\in H^0(X'_{g,\ell}, \P^{\otimes (-\ell)})$ is precisely the divisor
\begin{equation}\label{vanlocus}
\sum_{a=1}^{\lfloor \frac{\ell}{2}\rfloor} \Bigl(a\sum_{i=1}^{\ell-a} i\ \E_i^{(a)}+(\ell-a)\sum_{i=1}^{a-1} (a-i)\E_{\ell-a+i}^{(a)}\Bigr).
\end{equation}
Indeed, we already know that $\cD:=\mbox{supp}(Z(\Phi))$ can be expressed as a linear combination of exceptional divisors
$\sum_{a, i} c_i^{(a)}\E_i^{(a)}$, where $1\leq i\leq \ell-1$ and $1\leq a\leq \lfloor \frac{\ell}{2}\rfloor$, which is furthermore characterized by two sets of conditions:
\begin{enumerate}
\item For each $1\leq a \leq \lfloor \frac{\ell}{2}\rfloor$, one has that $c_1^{(a)}=a$\  and \ $c_{\ell-1}^{(a)}=\ell-a$.
\item For each $i\neq \ell-a$, one has $\mathrm{deg}(\mathcal{D}_{| E_i^{(a)}})=0$.
\end{enumerate}
The first condition expressed the fact that the vanishing order of $\Phi_{|C}$ at $p$ (respectively $q$) is equal to $a$ (respectively $\ell-a$). The second condition expresses the fact that the restriction of $\Phi$ on $E_i^{(a)}$
is the trivial morphism along each exceptional divisor $E_i^{(a)}$ with $i\neq \ell-a$,
inserted when passing from $X_{g, \ell}$ to $X'_{g, \ell}$. Taking into account that $\OO(\E_i^{(a)})=\OO(-2)$
on ${E_i^{(a)}}$,
the two conditions lead to a linear system of equation in the coefficients $c_i^{(a)}$, which we then solve to obtain the claimed formula (\ref{vanlocus}), which also computed $-\ell c_1(\P)$. Squaring this formula, pushing it forward while using formulas (1), (2), we obtain
$$\ell^2 {{{\mathsf u}}}_*(c_1^2(\P))=-2\Bigl(\sum_{i=1}^{\ell-a} a^2i^2+\sum_{i=1}^{a-1}(\ell-a)^2i^2+\sum_{i=1}^{\ell-a-1}a^2 i(i+1)+\sum_{i=1}^{a-1}(\ell-a)^2 i(i+1)\Bigr)\delta_0^{(a)},$$
which after manipulations leads to the claimed formula.
\end{proof}
\section{Syzygy jumping loci over $\rr_{g, \ell}$}
Following a principle already explained in \cite{F1}, one can construct tautological divisors on moduli spaces of curves defined in terms of syzygies of the parametrized objects. Such loci have a determinantal structure over an open subset of $\cR_{g, \ell}$, hence one can speak of their expected (co)dimension. When the expected codimension is $1$, one has virtual divisors on $\cR_{g, \ell}$.
\vskip 3pt
We begin by setting notation. For a smooth curve $C$, a line bundle $L$ and a sheaf $\F$ on $C$, we define the Koszul cohomology group $K_{p, q}(C;\F, L)$ as the cohomology of the complex
$$\bigwedge^{p+1} H^0(C, L)\otimes H^0(C, \F\otimes L^{\otimes(q-1)})\stackrel{d_{p+1, q-1}}\longrightarrow \bigwedge^p H^0(C, L)\otimes H^0(C, \F\otimes L^{\otimes q})\stackrel{d_{p, q}}\longrightarrow$$
$$ \stackrel{d_{p, q}}\longrightarrow \bigwedge^{p-1} H^0(C, L)\otimes H^0(C, \F\otimes L^{\otimes (q+1)}).$$
When $\F=\OO_C$, one writes $K_{p, q}(C, L):=K_{p, q}(C; \OO_C, L)$. When $\F$ is a line bundle, these Koszul cohomology groups can be interpreted in terms of syzygies of certain $0$-dimensional subschemes of $C$.
\begin{example}\label{zerodim}
Suppose $[C, \eta]\in \cR_{g, \ell}$  and $\eta\notin C_2-C_2$. Let $L:=K_C\otimes \eta\in \mbox{Pic}^{2g-2}(C)$ be the very ample paracanonical line bundle inducing an embedding $\phi_L:C\rightarrow \PP^{g-2}$. For an integer $1\leq k\leq \ell-1$, since $H^0(C, \eta^{\otimes k})=0$ one has that
$$K_{p, 1}(C; \eta^{\otimes k}, L)=\mbox{Ker}\Bigl\{\bigwedge^p H^0(L)\otimes H^0(L\otimes \eta^{\otimes k})\rightarrow \bigwedge^{p-1} H^0(L)\otimes H^0(L^{\otimes 2}\otimes \eta^{\otimes k})\Bigr\}.$$
This group can be viewed as parametrizing the syzygies of a $0$-dimensional scheme $Z\subset C\subset \PP^{g-2}$, where $Z\in |K_C\otimes \eta^{\otimes (1-k)}|$, see \cite{FMP} Proposition 1.6. Precisely,
\begin{equation}\label{mrcint}
b_{p+1, 1}(Z)=\mathrm{dim}\ K_{p, 1}(C; \eta^{\otimes k}, L).
\end{equation}
\end{example}

We explain Conjecture B and Theorems \ref{general bundles} and \ref{thmC} formulated in the introduction. We fix line bundles $\eta, \xi\in \mbox{Pic}^0(C)$, with $\xi\neq \OO_C$ and $\eta\otimes \xi\neq \OO_C$. To the paracanonical linear series $L:=K_C\otimes \eta$ we associate the kernel bundle defined via the exact sequence
$$0\longrightarrow M_L\longrightarrow H^0(C, L)\otimes \OO_C\longrightarrow  L\longrightarrow 0.$$
Using standard arguments, see e.g. \cite{FMP} Proposition 1.6, one has the identifications
\begin{equation}\label{mrc3}
K_{i, 1}(C; \xi, K_C\otimes \eta)=H^0\bigl(C, \bigwedge^i M_L\otimes L\otimes \xi\bigr),
\end{equation}
$$K_{i-1, 2}(C; \xi, K_C\otimes \eta)=
H^1\bigl(C, \bigwedge^i M_L\otimes L\otimes \xi\bigr).$$
The difference of the dimensions of the two cohomology groups is the Euler-Poincar\'e characteristic of a vector bundle on $C$, which in the case $\eta\neq \OO_C$ is equal to
\begin{equation}\label{differ}
\mbox{dim } K_{i, 1}(C; \xi, K_C\otimes \eta)-\mbox{dim } K_{i-1, 2}(C; \xi, K_C\otimes \eta)=(g-1){g-2\choose i}\Bigl(1-\frac{2i}{g-2}\Bigr).
\end{equation}
The naturality of the resolution of $\Gamma_C(\xi, L)$ as a $\mbox{Sym } H^0(C, L)$-module is equivalent to the vanishing for all $i$ of one of the groups $K_{i, 1}(C; \xi, L)$ or $K_{i-1, 2}(C; \xi, L)$, depending on the sign computed by (\ref{differ}). The resolution being minimal, it suffices to verify this for the values of $i$ when the expression (\ref{differ}), viewed as a function of $i$, changes sign. The resolution is pure, if the corresponding cohomology group vanishes for $i$ such that the expression in (\ref{differ}) is equal to zero.
\begin{lemma}\label{natural}
Let $C$ be a curve of genus $g$ and $\eta, \xi\in \mathrm{Pic}^0(C)-\{\OO_C\}$ such that $\eta\otimes \xi\neq \OO_C$.

\noindent (1) The resolution of $\Gamma_C(\xi, K_C\otimes \eta)$ as a $\mathrm{Sym } H^0(C, K_C\otimes \eta)$-module is natural if and only if
$$K_{\lfloor \frac{g-1}{2}\rfloor, 1}(C; \xi, K_C\otimes \eta)=0\  \ \mbox{ and }\ \ K_{\lfloor \frac{g-1}{2}\rfloor, 1}(C; (\eta\otimes \xi)^{\vee}, K_C\otimes \eta)=0.$$

\noindent (2) The resolution of $\Gamma_C(\xi, K_C)$ as a  $\mathrm{Sym }H^0(C, K_C)$-module is natural if and only if
$$ K_{\lfloor \frac{g}{2}\rfloor, 1}(C; \xi, K_C)=0 \ \ \mbox{ and } \ K_{\lfloor \frac{g}{2}\rfloor , 1}(C; \xi^{\vee}, K_C)=0.$$
\end{lemma}
\begin{proof} Assume that $g$ is odd and write $g=2i+1$. From (\ref{differ}), the resolution of the module $\Gamma_C(\xi, K_C\otimes \eta)$ is natural if and only if $K_{i, 1}(C; \xi, K_C\otimes \eta)=K_{i-2, 2}(C; \xi, K_C\otimes \eta)=0$. By the duality theorem \cite{G}, the dual of the last group is canonically isomorphic to
$$K_{i, 0}(C; K_C\otimes \xi^{\vee}, K_C\otimes \eta)=K_{i, 1}(C; (\eta\otimes \xi)^{\vee}, K_C\otimes \eta).$$
The case of the resolution of $\Gamma_C(\xi, K_C)$ (that is, $\eta=\OO_C$) is similar and we skip details.
\end{proof}
\vskip 3pt

We now describe the structure of the locus $\cU_{g, \ell}$ mentioned in Theorem \ref{thmC}. For $g=2i+1$, a result from \cite{FMP} provides an identification of cycles valid for each curve:
$$\Bigl\{\eta\in \mbox{Pic}^0(C): h^0(C, \bigwedge^i M_{K_C}\otimes K_C\otimes \eta)\geq 1\Bigr\}=C_i-C_i\subset \mbox{Pic}^0(C).$$
The right hand side denotes the $i$-th difference variety consisting of line bundles of the form $\OO_C(D-E)$, where
$D, E\in C_i$. This establishes the following equivalence
$$K_{i, 1}(C; \eta, K_C)\neq 0\Leftrightarrow \eta\in C_i-C_i.$$
After tensoring with $K_C\otimes \eta$ and taking cohomology in the short exact sequence
$$0\longrightarrow \bigwedge^i M_{K_C}\longrightarrow \bigwedge^i H^0(C, K_C)\otimes \OO_C\longrightarrow \bigwedge^{i-1} M_{K_C}\otimes K_C\longrightarrow 0,$$
we reformulate the last condition as follows: $[C, \eta]\in \cU_{g, \ell}$ if and only if the map
\begin{equation}\label{ugl}
\chi([C, \eta]):\bigwedge^i H^0(K_C)\otimes H^0(K_C\otimes \eta)\rightarrow H^0\bigl(C, \bigwedge^{i-1} M_{K_C}\otimes K_C^{\otimes 2}\otimes \eta\bigr)
\end{equation}
is an isomorphism. Since $h^0(\bigwedge^{i-1} M_{K_C}\otimes K_C^{\otimes 2}\otimes \eta)=\chi(\bigwedge^{i-1} M_{K_C}\otimes K_C^{\otimes 2}\otimes \eta)={2i\choose i-1}(4i+2)$, note that the locus $\cU_{g, \ell}$ can be defined as the degeneracy locus of a morphism between vector bundles of the same rank over $\mathsf{R}_{g, \ell}$, whose fibre over a point $[C, \eta]$ is precisely the map $\chi([C, \eta]$ defined above. Thus $\cU_{g, \ell}$ is a virtual divisor over $\cR_{g, \ell}$.

We show that  that $\chi([C, \eta])$ is an isomorphism for a general level curve, thus establishing Theorem \ref{thmC}.
In fact, we prove a more precise result valid for all genera\footnote{We are grateful to B. Poonen for pointing out to us an inaccuracy in the original version of this result.}:
\begin{theorem}\label{hyper}
Let $[C, p]\in \cM_{g, 1}$ be a general hyperelliptic curve of genus $g\geq 2$ together with a Weierstrass point. Then there exists a torsion point 
$\eta\in \mathrm{Pic}^0(C)[\ell]-\{\OO_C\}$ such that $$H^0\bigl(C, \eta((g-1)p)\bigr)=0.$$
\end{theorem}
First we explain how Theorem \ref{hyper} implies Theorem \ref{thmC}. Let $C$ be hyperelliptic, $A=\OO_C(2p)$ the hyperelliptic line bundle and $i:=\lfloor \frac{g}{2}\rfloor$. Then $M_{K_C}^{\vee}=A^{\oplus (g-1)}$, therefore
$$\bigwedge ^i M_{K_C}^{\vee}=\bigl(A^{\otimes i}\bigr)^{\oplus {g-1\choose i}}.$$
It follows from Lemma (\ref{natural}) that $K_{i, 1}(C; \eta, K_C)= 0$ if and only if $H^0(C, \eta((g-1)p))= 0$ when $g$ is odd, respectively
$H^0(C, \eta((g-2)p))=0$, when $g$ is even. Clearly, the statement of Theorem \ref{hyper} implies the naturality of the resolution $\Gamma_C(\eta, K_C)$ for all genera.

\vskip 3pt
\noindent \emph{Proof of Theorem \ref{hyper}.} Inductively we assume that $[C', \eta']\in \cR_{g-1, \ell}$ is a hyperelliptic level $\ell$ curve with a point $p'\in C'$ such that $h^0(C', \OO_{C'}(2p'))=2$ and $H^0\bigr(C', \eta'((g-2)p')\bigl)=0$. We consider a pointed elliptic curve $[E, p']\in \cM_{1, 1}$, which we attach to $C'$ at the point $p'$, then choose $p\in E-\{p'\}$ such that $2(p-p')\equiv 0$. Then $[C:=C\cup_{p'} E, p]\in \mm_{g, 1}$ is a degenerate hyperelliptic curve and $p\in C$ is a hyperelliptic Weierstrass point. This follows by exhibiting an admissible double covering $$\varphi:C\rightarrow (\PP^1)_1\cup_t (\PP^1)_2,$$ where $\varphi_{C'}:C'\rightarrow (\PP^1)_1$ is the hyperelliptic cover and $\varphi_{E}:E\rightarrow (\PP^1)_2$ is the double cover ramified at $p$ and $p'$. Note that $\varphi(p')=t\in (\PP^1)_1\cap (\PP^1)_2$.

\vskip 3pt
We choose an $\ell$-torsion point $\eta_E\in \mbox{Pic}^0(E)[\ell]$ such that $\eta_E\neq \OO_C\bigl((g-1)(p'-p)\bigr)$. Such a choice is possible since there at least two points of order $\ell$ inside $\mbox{Pic}^0(E)$. Then
$$[C'\cup_{p'} E,  \  \eta_{C'}, \ \eta_E]\in \rr_{g, \ell},$$
and we claim that this curve does not lie in the closure of the locus of  hyperelliptic level $\ell$ curves $[X_t, \eta_t]\in \cR_{g, \ell}$ with a  point $p_t\in X_t$ such that $h^0(X_t, \OO_{X_t}(2p_t))=2$ and $h^0(X_t, \eta_t((g-1)p_t))\geq 1$. Indeed, assuming this not to be true, applying a limit linear series argument we find non-zero sections
$$\sigma_{C'}\in H^0\bigl(C', \eta_{C'}\otimes \OO_{C'}((g-1)p')\bigr) \ \mbox{ and }\ \sigma_{E}\in H^0\bigl(E, \eta_{E}\otimes \OO_E((g-1)p)\bigr),$$
such that $\mbox{ord}_{p'}(\sigma_{C'})+\mbox{ord}_{p'}(\sigma_{E})\geq g-1$. Our assumption implies $\mbox{ord}_{p'}(\sigma_E)\leq g-2$, therefore $\mbox{ord}_{p'}(\sigma_{C'})\geq 1$, that is, $H^0\bigl(C', \eta_{C'}\otimes \OO_{C'}((g-2)p')\bigr)\neq 0$. This is a contradiction and completes the proof. \hfill $\Box$

\subsection{The syzygies of the twisted paracanonical module.}
We now show that for a general point $[C, \eta]\in \cR_{g, \ell}$, the module $\Gamma_C(\xi, K_C\otimes \eta)$ has a pure resolution for a general choice of $\xi\in \mbox{Pic}^0(C)$. The argument relies on specialization to bielliptic curves.

Suppose $E$ is an elliptic curve, $C$  a curve of genus $g$ and $f:C\rightarrow E$ a double cover ramified along the divisor $R\in C_{2g-2}$  and branched along the divisor $B\in E_{2g-2}$. Let $\delta\in \mbox{Pic}^{g-1}(E)$ denote the line bundle determining $f$, that is, $\delta^{\otimes 2}=\OO_E(B)$. In particular,
$$f_*\OO_C=\OO_E\oplus \delta^{\vee}, \ \ f^*(\delta)=\OO_C(R).$$

\begin{proposition}
Let $\epsilon \in \mathrm{Pic}^0(E)[\ell]$ an $\ell$-torsion point and $\eta:=f^*(\epsilon)\in \mathrm{Pic}^0(C)[\ell]$. Then
$$M_{K_C\otimes \eta}=f^*M_{\delta\otimes \epsilon}.$$
\end{proposition}
\begin{proof} From adjunction $K_C=f^*(\delta)$, therefore from the push-pull formula
$$H^0(C, K_C\otimes \eta)=f^*H^0(E, \delta\otimes \epsilon)\oplus f^* H^0(E, \epsilon)\cong H^0(E, \delta\otimes \epsilon).$$
In particular, pulling-back via $f$ the exact sequence defining the kernel bundle on $E$
$$0\longrightarrow M_{\delta\otimes \epsilon}\longrightarrow H^0(E, \delta\otimes \epsilon)\longrightarrow \delta\otimes \epsilon\longrightarrow 0,$$
we retrieve the sequence defining the kernel bundle on $C$,  that is, $M_{K_C\otimes \eta}=f^*(M_{\delta\otimes \epsilon})$.
\end{proof}

\begin{lemma}\label{stable_push}
The push-forward $f_*\xi$ of a general line bundle $\xi\in \mathrm{Pic}^0(C)$ is stable (respectively semistable) if $g$ is even (respectively odd).
\end{lemma}
\begin{proof} We treat only the case of even genus $g=2i+2$, the odd genus case being similar. The push-forward $f_* \xi$ is a rank two vector bundle on $E$ with $\mbox{det } f_*(\xi)=\mbox{Nm}_f(\xi)\otimes \delta^{\vee}$, in particular $\mbox{deg}(f_*\xi)=1-g$. Assume that $f_*\xi$ is not semistable. Then there exists a line subbundle $M\hookrightarrow f_*\xi$ with $\mbox{deg}(M)\geq \mbox{deg}(f_*\xi)/2$, that is, $\mbox{deg}(M)\geq -i$.
Then $H^0(C, \xi\otimes f^*M^{\vee})\neq 0$, hence we can write $\xi=f^*(M)(D)$, where $D$ is an effective divisor on $C$, with $\mbox{deg}(D)\leq 2i$. Counting parameters, line bundles on $C$ having this type depend on at most $2i+1=g-1$ parameters, hence they do not fill-up $\mbox{Pic}^0(C)$.
\end{proof}

Thus one obtains a rational map $f_*:\mbox{Pic}^0(C)\dashrightarrow \mathcal{U}_E(2, 1-g)$. By describing the differential of this map, it is easy to show that this map is dominant, see also \cite{B}.

We complete the proof of Theorem \ref{general bundles}. More precisely we prove the following:

\begin{theorem}\label{bielliptic}
Let $f:C\rightarrow E$ be a bielliptic curve and $\eta\in \mathrm{Pic}^0(C)[\ell]$ as above. Then
$$K_{\lfloor \frac{g-1}{2}\rfloor, 1}(C; \xi, K_C\otimes \eta)=0,$$
for a general $\xi\in \mathrm{Pic}^0(C)$. In particular, the resolution of $\Gamma_C(\xi, K_C\otimes \eta)$ is natural.
\end{theorem}
\begin{proof}
We treat only the case $g=2i+2$, the odd genus case being quite similar. Using identification (\ref{mrc3}) we write
$$K_{i, 1}(C; \xi, K_C\otimes \eta)=H^0\bigl(C, \bigwedge^i M_{K_C\otimes \eta}\otimes K_C\otimes \eta\otimes \xi\bigr)=H^0\bigl(E, \bigwedge^i M_{\delta\otimes \epsilon}\otimes \delta\otimes \epsilon\otimes f_*\xi\bigr).$$
It is well-known that the vector bundle $M_{\delta\otimes \epsilon}$ is stable, hence using also Lemma \ref{stable_push}, the vector bundle $\F:=\bigwedge^i M_{\delta\otimes \epsilon} \otimes \epsilon\otimes \delta\otimes f_*\xi$ is semistable. Futhermore, the choice of a general $\xi\in \mbox{Pic}^0(C)$ corresponds to the choice of a general $\F\in \mathcal{U}_E\bigl({g-2\choose i}, 0)$. Since on an elliptic curve, there exists precisely one semistable bundle of prescribed rank, degree zero and with a section, we find that $H^0(E, \F)=0$, which finishes the proof.
\end{proof}
Theorem \ref{bielliptic} admits the following reformulation in the case of even genus.
\begin{corollary}
For a general curve $[C, \eta]\in \cR_{g, \ell}$ of even genus, the vector bundle $\bigwedge^\frac{g-2}{2} M_{K_C\otimes \eta}$ admits a theta divisor.
\end{corollary}
Since $\mu\bigl(\bigwedge^{\frac{g-2}{2}} M_{K_C\otimes \eta}\bigr)=g-1\in \mathbb Z$ it makes sense to ask whether the vector bundle in question admits a theta divisor. Conjecture B is a more refined statement, predicting which line bundles $K_C\otimes \eta^{\otimes k}$ belong to the theta divisor of $\bigwedge^{\frac{g-2}{2}} M_{K_C\otimes \eta}$. Note that it is proved in \cite{FMP} that all powers $\bigwedge^i M_{K_C}$ admit a theta divisor for \emph{every} smooth curve $C$.

\section{Intersection theory on $\ssrr_{g, \ell}$}
The aim of this section is to describe the characteristic classes of tautological bundles on $\ssrr_{g, \ell}$ that are used to calculate the classes $[\zz_{g, \ell}]^{\mathrm{virt}}$, $[\ol{\cD}_{g, \ell}]^{\mathrm{virt}}$ and $[\ol{\cU}_{g, \ell}]$ respectively. For our purposes it will suffice to consider only level $\ell$ curves whose underlying stable model is irreducible. Let $\wt{\mathsf{R}}_{g, \ell}$ be the open substack of level $\ell$ curves whose underlying stable model is a $1$-nodal irreducible curves of arithmetic genus $g$. We denote
by $\mathsf{u}:\wt{\mathsf{X}}_{g, \ell}\rightarrow \wt{\mathsf{R}}_{g, \ell}$ the restriction of the universal level $\ell$ curve, by
$\P\in \sPic(\wt{\mathsf{X}}_{g, \ell})$ the tautological $\ell$th root bundle and by $\Phi:\P^{\otimes \ell}\rightarrow \OO_{\wt{\mathsf{X}}_{g, \ell}}$ the universal sheaf homomorphism. We shall use the Hodge bundle $\mathbb E:=\mathsf{u}_*(\omega_{\mathsf{u}})$ respectively the Prym-Hodge bundle $\mathbb E':=\mathsf{u}_*(\omega_{\mathsf{u}}\otimes \P)$. These are locally free sheaves over $\wt{\ssR}_{g, \ell}$ of ranks $g$ and $g-1$ respectively.

The following technical statement will be used to show that various tautological sheaves on $\wt{\ssR}_{g, \ell}$ are locally free.
\begin{proposition}\label{grauert1}
For a level curve $[X, \eta, \phi]\in \wt{\cR}_{g, \ell}$ and integers $b\geq 2$ and $0\leq j\leq \frac{g}{2}$, the following vanishing statements hold:
\begin{enumerate}
\item $H^1\bigl(X, \bigwedge^j M_{\omega_X}\otimes \omega_X^{\otimes b}\otimes \eta\bigr)=0.$
\item $H^1\bigl(X, \bigwedge^j M_{\omega_X\otimes \eta}\otimes \omega_X^{\otimes b}\otimes \eta^{\otimes (b-2)}\bigr)=0.$
\end{enumerate}
\end{proposition}
\begin{proof} As pointed out in Section 1, such a question can be studied at the level of root curves. To ease notation we identify $\mathrm{nor}([X, \eta, \phi])\in \mathrm{Root}_{g, \ell}$  and $[X, \eta, \phi]$. If $X$ is smooth, since $\mu(\wedge^j M_{\omega_X}\otimes \omega_X^{\otimes b}\otimes \eta)=b(2g-2)-2j\geq 2g-1$, the statements are a consequence of the semistability of $M_{\omega_X}$ and that of $M_{\omega_X\otimes \eta}$ respectively.

Assume now that $X$ is $1$-nodal and $\eta\in \mbox{Pic}(X)$ is locally free, that is, $[X, \eta, \phi]\in \Delta_0^{'}\cup \Delta_0^{''}$. Let $\mathrm{nor}:C\rightarrow X$ be the normalization map, with $\mathrm{nor}^{-1}(X_{\mathrm{sing}})=\{p, q\}$ and set $\eta_C:=\mbox{nor}^*(\eta)$.
Via the exact sequence $0\rightarrow \OO_X\rightarrow \mathrm{nor}_*(\OO_C)\rightarrow \mathrm{nor}_*(\OO_C)/\OO_X\rightarrow 0$, to show that
(ii) holds it suffices to show that
$$H^1\Bigl(C, \bigwedge^j M_{K_C(p+q)\otimes \eta_C}\otimes K_C^{\otimes b}\otimes \OO_C(bp+bq)\otimes \eta_C\Bigr)=0,$$
and
$$H^1\Bigl(C, \bigwedge^j M_{K_C(p+q)\otimes \eta_C}\otimes K_C^{\otimes b}\otimes \OO_C((b-1)p+(b-1)q)\otimes \eta_C\Bigr)=0.$$
This again is a consequence of the stability of the vector bundle $M_{K_C(p+q)\otimes \eta_C}$, which in turn follows from \cite{FL} Proposition 2.4.

\vskip 3pt
Assume now that $[X, \eta, \phi]\in \Delta_{0}^{\mathrm{ram}}$ and $X=C\cup_{\{p, q\}} E$, where $E\cong \PP^1$ and $[C, p, q]\in \cM_{g-1, 2}$. Furthermore $\eta_E=\OO_E(1)$ and $\mbox{deg}(\eta_C)=-1$. The kernel bundles $M_{\omega_X}$ and $M_{\omega_X\otimes \eta}$  have the following restrictions to the components of $X$:
$$M_{\omega_X|C}=M_{K_C(p+q)} \mbox{ and }\ M_{\omega_X| E}=\OO_E^{\oplus (g-1)},$$
respectively
$$M_{\omega_X\otimes \eta|C}=M_{K_C(p+q)\otimes \eta_C} \ \mbox{ and }\ M_{\omega_X\otimes \eta| E}=\OO_E(-1)\oplus \OO_E^{\oplus (g-3)}.$$
Via the Mayer-Vietoris sequence on $X$, a sufficient condition for (i) to hold is given by the following vanishing statements:
$$H^1\Bigl(\bigwedge^j M_{\omega_X|C}\otimes K_C^{\otimes b}\bigl(b(p+q)\bigr)\otimes \eta_C\Bigr)=0, \ \ H^1\Bigl(\bigwedge^j M_{\omega_X|C}\otimes K_C^{\otimes b}\bigl((b-1)(p+q))\otimes \eta_C\Bigr)=0,$$
as well as $H^1\bigl(E, \bigwedge^j M_{\omega_X|E}\otimes \eta_E\bigr)=0$ (note that $\omega_{X|E}=\OO_E$).
The vanishing on $E$ is immediate, whereas that on $C$ follows again by using that $M_{K_C(p+q)}$ and $M_{K_C(p+q)\otimes \eta_C}$ are semistable
and computing the slopes of the corresponding vector bundles whose first cohomology group is supposed to vanish. Statement (ii) is entirely similar and we skip details.
\end{proof}

To be able to define tautological sheaves over $\wt{\ssR}_{g, \ell}$, we consider the global \emph{kernel bundle} defined via the exact sequence over the universal curve $\wt{\mathsf{X}}_{g, \ell}$
$$0\longrightarrow \cM_{\mathsf{u}}\longrightarrow \mathsf{u}^*(\mathbb E)\longrightarrow \omega_{\mathsf{u}}\longrightarrow 0,$$
respectively the global \emph{Prym kernel bundle}
$$0\longrightarrow \cM'_{\mathsf{u}}\longrightarrow \mathsf{u}^*(\mathbb E')\longrightarrow \omega_{\mathsf{u}}\otimes \P \longrightarrow 0.$$

\subsection{Tautological sheaves.} We introduce tautological vector bundles over $\wt{\ssR}_{g, \ell}$ whose fibres are various Koszul cohomology groups. For integers $0\leq j\leq \frac{g}{2}$ and $b\geq 2$, as well as for $(j, b)=(0, 1)$, we define the sheaves
$$\mathbb{E}_{j, b}:=\mathsf{u}_*\bigl(\bigwedge^j \cM_{\mathsf{u}}\otimes \omega_{\mathsf{u}}^{\otimes b}\otimes \P\bigr) \mbox{ and }
\mathbb{F}_{j, b}:=\mathsf{u}_*\bigl(\bigwedge^j \cM'_{\mathsf{u}}\otimes \omega_{\mathsf{u}}^{\otimes b}\otimes \P^{\otimes (b-2)}\bigr).$$
Grauert's theorem used via Proposition \ref{grauert1} imply that $\mathbb{E}_{j, b}$ and $\mathbb{F}_{j, b}$ are locally free. Note that $\mathbb E_{0, 1}=\mathbb E'$. One can now carry a Grothendieck-Riemann-Roch calculation over the universal level $\ell$ curve $\mathsf{u}:\wt{\mathsf{X}}_{g, \ell}\rightarrow \wt{\ssR}_{g, \ell}$ and prove the following:
\begin{proposition}\label{grra}
For each integer $b\geq 1$, the following formula holds in $\mathsf{Pic}(\wt{\mathsf{R}}_{g, \ell})$:
\begin{enumerate}
\item $c_1(\mathbb{E}_{0, b})=\lambda+{b\choose 2}\kappa_1-\frac{1}{2\ell}\sum_{a=1}^{\lfloor \frac{\ell}{2}\rfloor} a(\ell-a)\delta_0^{(a)}$.
\item $c_1(\mathbb{F}_{0, b})=\lambda+{b\choose 2}\kappa_1-\frac{(b-2)^2}{2\ell}\sum_{a=1}^{\lfloor \frac{\ell}{2}\rfloor} a(\ell-a)\delta_0^{(a)}$.
\end{enumerate}
\end{proposition}

\subsection{The divisor $\ol{\cU}_{g, \ell}$.} The divisor $\cU_{g, \ell}$ has a scheme theoretic characterization via condition (\ref{ugl}). We extend this description over the boundary of $\wt{\cR}_{g, \ell}$ and compute the class of the resulting degeneracy locus, thus completing the proof of Theorem \ref{mrc1}.

By taking exterior powers in the sequence defining the sheaf $\cM_{\mathsf{u}}$ and then using Proposition \ref{grauert1}, we find that for
each $1\leq j\leq i-1$, one has the following exact sequences, which can be used to compute inductively the Chern classes of the vector bundles $\mathbb{E}_{j, b}$, starting from level $j=0$:
\begin{equation}\label{ind1}
0\longrightarrow \mathbb{E}_{j, i+1-j}\longrightarrow \bigwedge^j \mathbb E\otimes \mathbb{E}_{0, i+1-j}\longrightarrow \mathbb{E}_{j-1, i+2-j}\longrightarrow 0.
\end{equation}

With these ingredients, we can prove Theorem \ref{mrc1} and calculate $[\ol{\cU}_{g, \ell}]$:

\noindent
\emph{Proof of Theorem \ref{mrc1}.} We consider the morphism $\chi:\bigwedge^i \mathbb E\otimes \mathbb{E}_{0, 1}\rightarrow \mathbb{E}_{i-1, 2}$ of vector bundles over the stack $\wt{\mathsf{R}}_{g, \ell}$, which at the level of fibres is given by the map
$$\chi([X, \eta, \phi]):\bigwedge^i H^0(X, \omega_X)\otimes H^0(X, \omega_X\otimes \eta)\rightarrow H^0\bigl(X, \bigwedge^{i-1} M_{\omega_X}\otimes \omega_X^{\otimes 2}\otimes \eta\bigr).$$
The intersection of the degeneracy locus of $\chi$ with $\cR_{g, \ell}$ is precisely the divisor $\cU_{g, \ell}$, therefore the difference
$c_1(\mathbb{E}_{i-1, 2}-\bigwedge^i \mathbb{E} \otimes \mathbb{E}_{0, 1})-[\ol{\cU}_{g, \ell}]\in \Pic(\wt{\mathsf{R}}_{g, \ell})$ is a (possibly empty) effective class supported only on the boundary classes of $\wt{\cR}_{g, \ell}$. In particular, the class $c_1(\mathbb{E}_{g-1, 2}-\bigwedge^i \mathbb E\otimes \mathbb{E}_{0, 1})$ is effective. In order to compute it, we use (\ref{ind1}) and write that:
$$c_1\bigl(\mathbb{E}_{i-1, 2}-\bigwedge^i \mathbb{E}\otimes \mathbb{E}_{0, 1}\bigr)=\sum_{b=0}^i (-1)^{b+1} c_1\bigl(\bigwedge^{i-b} \mathbb E\otimes \mathbb{E}_{0, b+1}\bigr)=$$
$$\sum_{b=0}^i (-1)^{b+1}\Bigl[{g\choose i-b}\Bigl(\lambda+{b+1\choose 2}\kappa_1-\sum_{a=1}^{\lfloor \frac{\ell}{2}\rfloor}\frac{a(\ell-a)}{2\ell}\delta_0^{(a)}\Bigr)+(2b+1)(g-1){g-1\choose i-b-1}\lambda\Bigr],$$
which after routine calculations leads to Theorem \ref{mrc1}. \hfill $\Box$

\subsection{The virtual divisor $\ol{\cD}_{g, \ell}$.} We prove Theorem \ref{conjB}. Since the proof resembles the previous calculation, we succintly explain the main points. We set $g:=2i+2$ and recall that via (\ref{dgl}) that $[C, \eta]\in \cD_{g, \ell}$ if and only if the map
$$\bigwedge^i H^0(K_C\otimes \eta)\otimes H^0(C, K_C\otimes \eta^{\vee})\rightarrow H^0\bigl(C, \bigwedge^{i-1} M_{K_C\otimes \eta}\otimes K_C^{\otimes 2}\bigr)$$
is not an isomorphism. Equivalently, $\cD_{g, \ell}$ is the restriction to $\cR_{g, \ell}$ of the degeneracy locus of the vector bundle morphism
$\bigwedge^i \mathbb{E}'\otimes \mathbb{F}_{0, 1}\rightarrow \mathbb{F}_{i-1, 2}$.

\noindent \emph{Proof of Theorem \ref{conjB}}. For each integer $1\leq j\leq i-1$ we have the exact sequence on $\wt{\ssR}_{g, \ell}$
$$0\longrightarrow \mathbb{F}_{i-j, j+1}\longrightarrow \bigwedge^{i-j} \mathbb E'\otimes \mathbb{F}_{0, j+1}\longrightarrow \mathbb{F}_{i-j-1, j+2}\longrightarrow 0, $$
thus one writes that
$$[\ol{\cD}_{g, \ell}]^{\mathrm{virt}}=\sum_{j=0}^i (-1)^{j+1} c_1(\bigwedge ^{i-j} \mathbb E'\otimes \mathbb{F}_{0, j+1})=$$
$$\sum_{j=0}^i (-1)^{j+1}\Bigl[(g-1)(2j+1){g-2\choose i-j-1}c_1(\mathbb E')+{g-1\choose i-j}c_1(\mathbb{F}_{0, j+1})\Bigr].$$
Using Proposition \ref{grra} and that $\mathbb E^{'}=\G_{0, 1}$, we finish the proof after some calculations.
\hfill $\Box$

\subsection{The Prym-Green divisor.}
Recall that $\mathbb{E}'$ is the Prym-Hodge bundle with fibres
$\mathbb{E}'([X, \eta, \phi])=H^0(X, \omega_X\otimes \eta)$.
For $b\geq 1$, we set $\mathbb{H}_{0, b}:=\mbox{Sym}^b \mathbb{E}'$. Then for $a\geq 1$, we define inductively the sheaves $\mathbb{H}_{a, b}$ via the
following exact sequences over $\wt{\mathsf{R}}_{g, \ell}$:
\begin{equation}\label{hi}
0\longrightarrow \mathbb{H}_{a, b}\longrightarrow \bigwedge^a \mathbb{E}'\otimes \mathbb{H}_{0, b}\longrightarrow \mathbb{H}_{a-1, b+1}\longrightarrow 0.
\end{equation}
Note that $\mathbb{H}_{a, b}$ is locally free and its fibre over a point $[X, \eta, \phi]$ inducing a paracanonical map $\phi_{L}:X\rightarrow \PP^{g-2}$, where $L:=\omega_X\otimes \eta$  is canonically identified with the space $H^0(\PP^{g-2}, \Omega_{\PP^{g-2}}^1(a+b))$.

\vskip 3pt
For each $b\geq 1$, we also define the sheaf $\mathbb{G}_{0, b}:=\mathsf{u}_*(\omega_{\mathsf{u}}^{\otimes b}\otimes \P^{\otimes b})$. Note that $\mathbb{G}_{0, 1}=\mathbb{E}'$ and there exist sheaf homomorphisms $\varphi_{0, b}:\mathbb{H}_{0, b}\rightarrow \mathbb{G}_{0, b}$, which fibrewise over $[X, \eta, \phi]\in \rr_{g, \ell}$ correspond to the multiplication maps of global sections
$$\mbox{Sym}^b H^0(X, \omega_X\otimes \eta)\rightarrow H^0(X, \omega_X^{\otimes b}\otimes \eta^{\otimes b}).$$
Inductively, for each $a\geq 1$, we define sheaves $\mathbb{G}_{a, b}$ over $\wt{\mathsf{R}}_{g, \ell}$ via the exact sequences
\begin{equation}\label{gi}
0\longrightarrow \mathbb{G}_{a, b}\rightarrow \bigwedge^a \mathbb{E}'\otimes \mathbb{G}_{0, b}\rightarrow \mathbb{G}_{a-1, b+1}\longrightarrow 0,
\end{equation}
where the right exactness of the sequence (\ref{gi}) is to be soon justified. Inductively, one also constructs sheaf homomorphisms
$$\varphi_{a, b}:\mathbb{H}_{a, b}\rightarrow \mathbb{G}_{a, b},$$ fibrewise given by restriction of twisted forms on projective space
$$\varphi_{a, b}([X, \eta, \phi]): H^0(\PP^{g-2}, \bigwedge^a M_{\PP^{g-2}}(b))\rightarrow H^0(X, \bigwedge^a M_{L}\otimes L^{\otimes b}).$$
The right exactness of the sequence (\ref{gi}) is established via Grauert's theorem by the following:
\begin{proposition}
Let $[X, \eta, \beta]\in \wt{\cR}_{g, \ell}$ a twisted level curve. Then for all integers $a\geq 0$, $b\geq 2$, the vanishing
$H^1\bigl(X, \bigwedge^a M_{L}\otimes L^{\otimes b}\bigr)=0$ holds.
\end{proposition}
\begin{proof} It follows closely \cite{FL} Proposition 3.2, where the corresponding statement when $\ell=2$ is established.
\end{proof}
We now compute the Chern classes of the above defined tautological bundles (this matches
the main theorem of \cite{CGRR} for $s=b\ell$ via $\frac12B_2(\frac{s}{\ell})=\binom{b}{2}+1/6$;
we refer to Section 3.2, (44) for the explicit example in degree $1$).
\begin{proposition}\label{gj}
For $b\geq 1$ we have the following formula in $\Pic(\wt{\mathsf{R}}_g)$:
$$c_1(\mathbb{G}_{0, b})=\lambda+{b\choose 2}\kappa_1-\frac{b^2}{2}\sum_{a=1}^{\lfloor \frac{\ell}{2}\rfloor} \frac{a(\ell-a)}{\ell}\delta_0^{(a)}.$$
\end{proposition}
\begin{proof} We apply Grothendieck-Riemann-Roch for the universal curve $\mathsf{u}:\wt{\mathsf{X}}_{g, \ell}\rightarrow \wt{\mathsf{M}}_g$ and the sheaf $\mathbb{G}_{0, b}$. Noting that $R^1 \mathsf{u}_*(\omega_{\mathsf{u}}^{\otimes b}\otimes \P^{\otimes b})=0$, we write:
$$\mathrm{ch}(\mathbb{G}_{0, b})=\mathsf{u}_*\Bigl[\Bigl(1+ b\ c_1(\omega_{\mathsf{u}}\otimes \P)+\frac{b^2}{2} c_1^2(\omega_{\mathsf{u}}\otimes \P)+\cdots\Bigr)\cdot\Bigl(1-\frac{c_1(\omega_{\mathsf{u}})}{2}+\frac{c_1^2(\omega_{\mathsf{u}})+[\mathrm{Sing}(\mathsf{u})]}{12}\Bigr)\Bigr],$$
where $\mathrm{Sing}(\mathsf{u})\subset \wt{X}_{g, \ell}$ denotes the codimension $2$ singular locus of the $\mathsf{u}$, and clearly $\mathsf{u}_*([\mathrm{Sing}(\mathsf{u})])=\mathsf{f}^*(\delta_0)$, where $\delta_0\in \Pic(\wt{\mathsf{M}}_g)$.
Using Mumford's formula \cite{HM} $\mathsf{u}_*(c_1^2(\omega_{\mathsf{u}}))=12\lambda-\delta$ as well as Proposition \ref{features}, we obtained the claimed formula by evaluating the push-forward under $\mathsf{u}$ of the quadratic terms.
\end{proof}
We now compute the virtual class of the Prym-Green divisor.

\noindent
\emph{Proof of Theorem \ref{prymgreenclass}.} We set $g=2i+6$ and observe that the locus $\cZ_{g, \ell}$ of smooth level curves
$[C, \eta]\in \cR_{g, \ell}$ is the degeneracy locus of the morphism $\varphi_{i, 2}: \mathbb{H}_{i, 2 |\ssR_{g, \ell}}\rightarrow \mathbb{G}_{i, 2 |\ssR_{g, \ell}}$. Whenever Conjecture A is true, the class $[\zz_{g, \ell}]^{\mathrm{virt}}$ is effective, and differs from the class of the closure $\zz_{g, \ell}$ by a (possibly empty) effective combination of boundary divisors.
We now compute $[\zz_{g, \ell}]^{\mathrm{virt}}:=c_1(\mathbb{G}_{i, 2}-\mathbb{H}_{i, 2})$ as follows:
$$c_1(\mathbb H_{i, 2})=\sum_{j=0}^i (-1)^j c_1\bigl(\bigwedge^{i-j} \mathbb{H}_{0, 1}\otimes \mathrm{Sym}^{j+2}(\mathbb{H}_{0, 1})\bigr),$$
and $$c_1(\mathbb G_{i, 2})=\sum_{j=0}^i (-1)^j {g-1\choose i-j}c_1(\mathbb{G}_{0,i+2})+\sum_{j=0}^i (-1)^j (g-1)(2j+3){g-2\choose i-j-1}c_1(\mathbb E').$$
Using Proposition \ref{gj} we can finish the proof. \hfill $\Box$

\begin{remark} In the course of the proof of Theorem \ref{koddim} it is important to note that for any $\ell\geq 2$ and $g\leq 23$, in order to establish that the class $K_{\rr_{g, \ell}}$ is effective (respectively big), it suffices to prove the seemingly weaker statement that the restriction $K_{\wt{\cR}_{g, \ell}}$ is effective (respectively big). Precisely, if $\cD$ is an effective divisor on $\cR_{g, \ell}$ such that
$$[\ol{\cD}]=s  \lambda-2(\delta_0^{'}+\delta_0^{''})-\sum_{a=1}^{\lfloor \frac{\ell}{2}\rfloor} b_0^{(a)}\delta_0^{(a)} -\sum_{i=1}^{\lfloor \frac{g}{2}\rfloor} (b_i\ \delta_i+b_{g-i}\ \delta_{g-i}+
b_{i: g-i} \delta_{i: g-i})\in \mathsf{Pic}(\ssrr_{g, \ell}), $$
with $s\leq 13$ and $b_0^{(a)}\geq \ell+1$ for $a=1, \ldots, \lfloor \frac{\ell}{2}\rfloor$, then $K_{\rr_{g, \ell}}-[\ol{\cD}]\in \mbox{Eff}(\ssrr_{g, \ell})$ is an effective class, that is $b_{g-1}\geq 3$ and $b_i, b_{i:g-i}\geq 2$, for all $i=1, \ldots, g-1$. The proof uses pencils of level curves on $K3$ surfaces and is similar to \cite{FL} Proposition 1.2.
For $1\leq i\leq \mbox{min}\{\frac{g}{2}, 11\}$, we fix a general  curve $[C_2, p]\in \cM_{g-i, 1}$, an $\ell$-torsion point $\eta_{C_2}\in \mbox{Pic}^0(C_2)$  and a moduli curve $B_i\subset \mm_{i, 1}$, induced by a Lefschetz pencil $\{(C_t, p)\}_{t\in \PP^1}$ of pointed curves on genus $i$ on a fixed $K3$ surface $S\subset \PP^i$, the marked point $p$ being one of the base points of the pencil. We construct three $1$-dimensional families filling-up the divisors $\Delta_{g-i}, \Delta_i$ and $\Delta_{i:g-i}$ as follows. Firstly, $A_{g-i}\subset \Delta_{g-i}$ consists of level curves $\bigl\{[C_t\cup_p C_2, \eta_{C_t}=\OO_{C_t}, \eta_{C_2}]\bigr\}_{t\in \PP^1}\subset \rr_{g, \ell}$. Then $A_i\subset \Delta_i$ parametrizes level curves $\bigl\{[C_t\cup _p C_2, \ \eta_{C_t}\in \ol{\mathrm{Pic}}^0(C_t)[\ell], \ \OO_{C_2}]\bigr\}_{t\in \PP^1}$. Finally,  $A_{i:g-i}\subset \Delta_{i:g-i}$ consists of level curves  $\bigl\{[C_t\cup _p C_2, \ \eta_{C_t}\in \ol{\mbox{Pic}}^0(C_t)[\ell], \ \eta_{C_2}]\bigr\}_{t\in \PP^1}$. For $i\neq 10$, the curves $A_i, A_{g-i}$ and $A_{i:g-i}$ fill-up the respective boundary divisors in $\rr_{g, \ell}$. By imposing the conditions $A_i\cdot \ol{\cD}\geq 0$, $A_{g-i}\cdot \ol{\cD}\geq 0$ and $A_{i:g-i}\cdot \ol{\cD}\geq 0$ and computing the actual intersection numbers via \cite{FP}, we obtain the desired bounds on the boundary coefficients of $[\ol{\cD}]$. The case $g=10$ is a little special but can be handled in a similar way, cf. \cite{FP} Theorem 1.1 (b).
\end{remark}

\section{Syzygy computations with nodal curves}

In this section we explain how to verify computationally Conjectures A and B for small $g$ and bounded level $\ell$ using nodal curves.

Let $C$ be a rational $g$-nodal curve with normalization $\mathrm{nor}: \PP^1 \to C$ and denote by $\{P_j, Q_j\}_{j=1}^g$ the preimages of the nodes of $C$.
A line bundle $L\in \mbox{Pic}^d(C)$ is given by an isomorphism $\mathrm{nor}^*(L)\isom \OO_{\PP^1}(d)$ and gluing data between residue class fields
$$a_j: \OO_{\PP^1}(d)  \tensor \kappa(P_j) {\isom} \OO_{\PP^1}(d)  \tensor \kappa(Q_j).$$ In particular, $\Pic^0(C) \isom \GG_m \times \ldots \times \GG_m$ is a $g$-dimensional torus.

Let $K[x_0,x_1]$ be the homogeneous coordinate ring of $\PP^1$, and let $z=\frac{x_1}{x_0}$ be the affine coordinate on the chart $U_0\isom \Aa^1$. We assume that all preimage points of the nodes are contained in $U_0$, say $P_j=(1:p_j)$ and $Q_j=(1:q_j)$. Then
$$
H^0(C, L) \isom  \{ f \in K[z]_{\le d}: a_j f(p_j)=  f(q_j) \mbox{ for } \ i=1, \ldots, g \}
$$
can be identified with  the space of polynomials $f$ of degree at most $d$ whose values in $p_j$ and $q_j $ differ by the factor $a_j \in K^*$.
In terms of coefficients of polynomials,  the space $H^0(C, L) \subset H^0(\PP^1,\OO_{\PP^1}(d))$ is the solution space of a homogeneous system of equations.

\begin{example} A basis of sections of the dualizing sheaf  $\omega_C$ is given by
$$\omega_j = \frac{dz}{(z-p_j)(z-q_j)}, \hbox{ for } j=1,\ldots,g,$$
and the canonical map is induced by
 $$
 \phi_{\omega_C}:\PP^1 \to \PP^{g-1}, \ \ z \mapsto \Bigl( \prod_{i \not=1}(z-p_i)(p-q_i): \ldots : \prod_{i \not=g}(z-p_i)(z-q_i)\Bigr).
 $$
Thus $\mathrm{nor}^*(\omega_C)^=\OO_{\PP^1}(2g-2)$ and the the canonical multipliers are $a^{\mathrm{can}}_j= \prod_{i\not= j} \frac{(q_j-p_i)(q_j-q_i)}{(p_j-p_i)(p_j-q_i)}$.
\end{example}

\begin{example} Suppose $A, B \in \Pic(C)$ correspond to the pairs $\bigl(\OO_{\PP^1}(d),(a_1,\ldots,a_g)\bigr)$ and $\bigl(\OO_{\PP^1}(e),(b_1,\ldots,b_g)\bigr)$ respectively, then
$A \tensor B$ is given by $\bigl(\OO_{\PP^1}(d+e),(a_1b_1,\ldots,a_gb_g)\bigr)$.
\end{example}

\begin{example} Let $\eta \in \Pic^0(C)$ be a line bundle corresponding to $(\OO_{\PP^1},(\zeta_1,\ldots,\zeta_g))$.
Then $\eta$ is a torsion bundle, if and only if all multipliers $\zeta_j$ are roots of unity. In particular, one obtains a level $\ell$ paracanonical curve by altering some or all of the
multipliers $a_j^{\mathrm{can}}$ by a primitive $\ell$th root of unity.
\end{example} \medskip

\noindent
We verify the Prym-Green Conjecture following the following steps:
\begin{enumerate}

\item If $\ell=2$, we take $r=1$ and choose a prime $p$ of moderate size, for instance $10^4 < p <3\cdot 10^4$.  In case $ \ell>2$ we choose
an integer $r$ and a prime $p$ such that $r$ represents a primitive  $\ell$th root of unity in $K=\FF_p$. So $p$ is one of prime factors of $r^\ell-1$.

\item Randomly pick $2g$ points $P_1, Q_1, \ldots, P_g, Q_g\in \PP^1(K)$ and compute the canonical multipliers $a_j^{\mathrm{can}}$ for $j=1, \ldots, g$.
\item  Alter all (or some) multiplier $a_j^\eta= r a_j^{\mathrm{can}}$ and compute a basis $f_0,\ldots,f_{g-2}$ of the Prym canonical system $H^0(C,\omega_C \tensor \eta) \subset H^0(\PP^1, \OO_{\PP^1}(2g-2))$. By Riemann-Roch, this space is $(g-1)$-dimensional.

\item Compute the kernel $I_C$ of the map $K[y_0,\ldots,y_{g-2}] \to K[x_0,x_1]$ given by $y_j \mapsto f_j$ and its free resolution up to the appropriate step.
\end{enumerate}

A {\it Macaulay2} code which does this job can be found in the package \emph{\href{http://www.math.uni-sb.de/ag/schreyer/home/computeralgebra.htm}{NodalCurves.m2 }} available at  \emph{\href{http://www.math.uni-sb.de/ag/schreyer/home/computeralgebra.htm}{http://www.math.uni-sb.de/ag/schreyer/home/computeralgebra.htm}}.

\begin{proposition}\label{pgverif}
The Prym-Green Conjecture holds for all even genera with $g\le 18$ and $\ell\leq 5$ over a field $K$  of characteristic $0$ with the possible exception of the cases $(g, \ell)=(8,2)$ or $(16,2)$.
\end{proposition}
\begin{proof} By the {\it Macaulay2} computation documented above, the statement holds true for some $g$-nodal rational curve defined over a finite field $\FF_p$.
This computation can be viewed as the reduction mod $p$ of the same computation of an example defined over $\mathbb Z$. By semicontinuity of Betti numbers in families, then there exists a $g$-nodal rational curve defined over $\QQ$ with a pure resolution. Applying again semicontinuity, a general pair $[C,\eta]\in \cR_{g, \ell}$ has a pure resolution of its paracanonical embedding.
\end{proof}

In more concrete terms, Proposition \ref{pgverif} says that a general level $\ell$ paracanonical curve of genus $10$ has the following syzygies
\begin{verbatim}
        0  1  2   3   4   5  6  7
 total: 1 18 42 126 210 162 63 10
     0: 1  .  .   .   .   .  .  .
     1: . 18 42   .   .   .  .  .
     2: .  .  . 126 210 162 63 10
\end{verbatim}
for all $2 \le \ell \le 5$. The relevant computation in this example takes about  $1.80$ seconds. Using $2$ minutes of cpu we can extend this result for  $g=10$  to all levels $\ell \le 30$.

For larger genus the approach has to to modified to be still computationally feasible. Recall \cite{E} that for a graded Cohen-Macaulay module $M$ of dimension $d$ over a standard graded polynomial ring $S$, if $x_1,\ldots,x_d \in S_1$ is an $M$-sequence, the  minimal free resolution of the artinian quotient $A=M/\lideal x_1,\ldots, x_d \rideal M$ over the polynomial ring $T=S/\lideal x_1,\ldots, x_d \rideal$ is obtained from the minimal free resolution of $M$ as an $S$-module by tensoring
with $T$. In particular the graded Betti numbers
$$b_{ij}(M)=\dim_K \Tor^S_i(M,K)_j = \dim_K \Tor^T_i(A,K)_j=b_{ij}(A)$$
coincide. We also can use duality  \cite[Theorem 21.15]{E} : Let $\omega_S \isom S(-\dim S)$ denote the dualizing module of S, and $\omega_M := \Ext^{\codim M}_S(M,\omega_S)$. If $M$ is Cohen-Macaulay, then the minimal free resolution of $\omega_M$ is obtained by applying $\Hom_S(-,\omega_S)[-\codim M]$ to the minimal free resolution of $M$. In particular,
$$\dim_K \Tor^S_i(M,K)_j = \dim_K \Tor^S_{\codim M-i}(\omega_M,K)_{\dim S -j}.$$

Finally, we can use the Koszul complex, that is, the minimal free resolution of $K$ as an $S$-module in order to compute $\Tor^S(M,K)$. We use all three options.

\begin{enumerate}
\item Compute a basis $s_0,\ldots,s_{g-2}$ of $H^0(C,K_C\tensor \eta) \subset K[x_0,x_1]_{2g-2}$ and $\omega_1,\ldots, \omega_{g}$ of $H^0(C,K_C) \subset K[x_0,x_1]_{2g-2}$.
\item Check that $s_{g-3},s_{g-2}$ have no common zero in $\PP^1$, i.e. they correspond to a $R$-sequence for $R=K[y_0,\ldots,y_{g-2}]/I_C \isom \Gamma_C(\OO_C,K_C\tensor \eta)$.
\item Compute representatives in $K[x_0,x_1]$ of a $K$-basis for the artinian reduction $A= \omega_R/\lideal(y_{g-3},y_{g-2}\rideal \omega_R$. Note that $A$ is a graded  artinian
$T=K[y_0,\ldots y_{g-4}]$-module
with Hilbert series $H_A(t) =\sum_{d} \dim A_d t^d=g+(g-3)t+t^2$.
\item Set $m=\frac{g}{2}$. Substitute $s_0,\ldots s_{g-4}$ into the $m^{th}$ Koszul matrix on $y_0,\ldots,y_{g-4}$ and compute the tensor product with the
$1 \times g$ matrix $(\omega_1,\ldots,\omega_{g})^t$. The result is a $ { g-3 \choose \frac{g}{2}-1} \times g{g-3 \choose \frac{g}{2}}$ matrix with entries in $H^0(C,\omega_C^2\tensor \eta) \subset K[x_0,x_1]_{4g-4}$.
\item Reduce the entries module the ideal $\lideal s_{g-3},s_{g-2} \rideal\lideal \omega_1,\ldots,\omega_{g} \rideal \subset K[x_0,x_1]$. The result is now a matrix $M_{\mathrm{pol}}$ of polynomials with entries in the $(g-3)$-dimensional space of representatives of $A_2$.
\item Compute the  $(g-3) { g-3 \choose \frac{g}{2}-1} \times g{g-3 \choose \frac{g}{2}}$ coefficient matrix of $M_{\rm field}$ with values in $K$. The kernel of $M_{\mathrm{field}}$ is isomorphic to $$\Tor^m_S(\omega_R,K)_m \isom \Tor^{g-3-m}_S(R,K)_{g-1-m} \isom K_{g-3-m,2}(C,K_C \tensor \eta),$$ whose dimension equals
$\dim K_{\frac{g}{2}-2,1}(C,K_C\tensor \eta)$ since $g-3-m+1=\frac{g}{2}-2$.
\item Use the finite field linear algebra subroutines/packages \emph{\href{http://www-ljk.imag.fr/membres/Jean-Guillaume.Dumas/FFLAS/}{FFLAS}} \cite{DGP} to compute the rank of $M_{\mathrm{field}}$ over the finite field $K=\FF_p$.
\end{enumerate}
\vskip 3pt

Our approach to the torsion bundle conjecture (Conjecture B) for bounded $g$ and $\ell$ is similar. Again we work with a $g$-nodal curve $C$ and carry out the following steps:

\begin{enumerate}

\item Compute bases of $s_0,\ldots, s_{g-2}$ of $H^0(C,K_C \tensor \eta) \subset K[x_0,x_1]_{2g-2}$ and $t_0,\ldots,t_{g-3}$ of
 $H^0(C,K_C \tensor \eta^k) \subset K[x_0,x_1]_{2g-2}$ and the kernel of the multiplication map $$\mu: H^0(C,K_C \tensor \eta) \tensor H^0(C,K_C \tensor \eta^{\otimes k}) \to H^0(C,K^2_C \tensor \eta^{\otimes (k+1)}).$$ Since the map is surjective,  $\mbox{Ker}(\mu)$ is $(g-1)^2-(3g-3)=(g-1)(g-4)$ dimensional.

\item $\mbox{Ker}(\mu)$ can be reinterpreted as the linear presentation matrix $$\phi: S^{g^2-5g+4}(-2) \to S^{g-1}(-1)$$ of $M=\Gamma_C(\eta^{\otimes k}, K_C \tensor \eta)$ as an $S=K[y_0,\ldots,y_{g-2}]$ module. We wish to compute $\Tor^S_{\frac{g}{2}-1}(M,K)_{\frac{g}{2}} \isom K_{\frac{g}{2}-1,1}(\eta^{\otimes k}, K_C \tensor \eta)$.

\item Check that $s_{g-3},s_{g-2}$ have no common zero in $\PP^1$ and  compute basis for the  artinian reduction $B= M/\lideal y_{g-3},y_{g-2}\rideal M $. This time $B$ is a graded  artinian
$T=K[y_0,\ldots y_{g-4}]$-module
with Hilbert series $H_B(t) =(g-1)t+(g-1)t^2$.

\item Reinterpret the tensor product of the $(\frac{g}{2}-1)^{\mathrm{th}}$ Koszul matrix on $y_0, \ldots, y_{g-4}$ with $B$ as a
$(g-1){g-3 \choose \frac{g}{2}}$ square matrix with entries in the ground field, and compute its rank. The dimension of the kernel is the desired Betti number.

\end{enumerate}

\begin{proposition}
Conjecture B  hold over a field of characteristic zero for all $6 \le g\le 16$ and all primitive $\ell$-torsion bundle
with $\ell = 3$.
\end{proposition}

In more concrete terms the result says for $g=10$ that a general pair $[C,\eta] \in \rr_{g,\ell}$ the module $\Gamma_C(C; \eta^{\otimes k}, K_C\tensor \eta)$ for $1 \le k \le \ell-2$ has syzygies
\begin{verbatim}
       0  1   2   3   4   5  6 7
total: 9 54 126 126 126 126 54 9
    0: 9 54 126 126   .   .  . .
    1: .  .   .   . 126 126 54 9
\end{verbatim}
unless $\ell=2k+1$, in which case we have instead
\begin{verbatim}
       0  1   2   3   4   5  6 7
total: 9 54 126 126 126 126 54 9
    0: 9 54 126 126   1   .  . .
    1: .  .   .   1 126 126 54 9
\end{verbatim}

\begin{proof} By our computation, Conjecture B holds for an $g$-nodal example over a finite field. For documentation of the computation based on the {\it Macaulay2} package \emph{\href{http://www.math.uni-sb.de/ag/schreyer/home/computeralgebra.htm}{NodalCurves.m2}} see \emph{\href{http://www.math.uni-sb.de/ag/schreyer/home/computeralgebra.htm}{http://www.math.uni-sb.de/ag/schreyer/home/computeralgebra.htm}}. To complete the proof with a semi-continuity argument, we have to show that in the exceptional cases the Betti number cannot become zero.

The extra syzygy are explained as follows. The resolution of $\eta^{\otimes k}$ is  self-dual if and only if
$$\begin{array}{lllll}
\eta^{\otimes k}
 &\cong &
  \sExt^{g-3}_\OO( \eta^{\otimes k},\OO(-g))
  &\cong  &
  \sExt^{g-3}_\OO( \eta^{\otimes k},\omega_{\PP^{g-2}} )(-1) \\
 &\cong &
  \sHom(\eta^{\otimes k},K_C)\tensor K_C^{-1} \tensor \eta^{-1}
   &\cong &
    \eta^{\otimes (-k-1)}
\end{array}
$$
that is, $\eta^{\otimes (2k+1)}=\OO_C$. In this case the artinian reduction, $ B=B_1 \oplus B_2$ is self-dual as well, that is, $B_2= B_1^{\vee}$ and the multiplication map
 $V \tensor B_1 \to B_2$ with $V=K[y_0,\ldots,y_{g-4}]_1$ corresponds to a symmetric $(g-1)$ square matrix with entires in $V^{\vee}$.
 The relevant Koszul cohomology map for $B$
 $$
\kappa_B:\Lambda^{(g-2)/2} V \tensor B_1 \to \Lambda^{(g-4)/2} V \tensor B_2
$$
is obtained from the Koszul map
$$
\kappa:\Lambda^{(g-2)/2} V \to \Lambda^{(g-4)/2} V \tensor V
$$
by tensor product  with $\mu' : B_1 \to B_2 \tensor V^{\vee}$ and contraction $V^{\vee} \tensor V \to K$. We now apply the following well-known fact:

\begin{lemma}\label{middleKoszulMatrix}The middle Koszul matrix
$$\kappa':\Lambda^{m+1} V \tensor T \to \Lambda^{m} V \tensor T(1)$$
of a polynomial ring $T= \Sym V$  in $2m+1$ variables is symmetric if $m \equiv 0 \mod 2$. Otherwise $\kappa'$ is skew-symmetric.
\end{lemma}

Thus in case $\ell=2k+1$ the matrix $\kappa_B$ is as tensor product of matrices with symmetry properties, skew-symmetric precisely in case  $\frac{g-4}{2} \equiv 1 \mod 2$ i.e. $g \equiv 2 \mod 4$.
If in addition ${g-3 \choose \frac{g}{2}-1}$ is odd, then $\kappa_B$ is a skew-symmetric matrix of odd size, and $\kappa_B$ cannot have maximal rank.
This completes the argument and justifies the exceptions in the statement of Conjecture B.
\end{proof}

\noindent {\it Proof of Lemma \ref{middleKoszulMatrix}.}
Under the pairing $\Lambda^{m}V \times \Lambda^{m+1}V \to \Lambda^{2m+1} V \isom K$ the matrix $\kappa'$
corresponds to the composition
$$
\psi:\Lambda^{m+1}V \tensor\Lambda^{m+1}V \to V \tensor \Lambda^m V \tensor \Lambda^{m+1} V \to V \tensor \Lambda^{2m+1}V.
$$
Hence for basis elements $y_I=y_{i_0} \wedge \ldots \wedge y_{i_{m}},y_J=y_{j_0} \wedge \ldots \wedge y_{j_{m}} \in \Lambda^{m+1}V$ for $I,J \subset \{0,\ldots,2m\}$ we have that
$\psi(y_I \tensor y_J)=0$ unless $I \cup J=\{0,\ldots,2m\}$, which means that $I \cap J$ consists of precisely one element. Furthermore for disjoint sets $I'=\{i_1,\ldots,i_m\}$ and $J'=\{j_1,\ldots,j_m\}$ we compute that
$$
\psi(y_{i_0} \wedge y_{I'} \tensor y_{i_0} \wedge y_{J'})
 = y_{i_0} \tensor y_{I'} \wedge y_{i_0} \wedge y_{J'}
 = (-1)^m y_{i_0} \tensor y_{i_0} \wedge y_{I'} \wedge y_{J'}$$
 $$= y_{i_0} \tensor y_{i_0} \wedge y_{J'} \wedge y_{I'}
 = (-1)^m y_{i_0} \tensor y_{J'} \wedge y_{i_0} \wedge y_{I'}
 = (-1)^m \psi(y_{i_0} \wedge y_{J'} \tensor y_{i_0} \wedge y_{I'})$$ as claimed. \qed

\begin{remark} Notice that for a field of characteristic 2 the matrix $\psi$ is symmetric with zeroes on the diagonal. Thus for $\ell=2k+1$ and  ${g-3 \choose \frac{g}{2}-1} \equiv 1 \mod 2$ a torsion bundle has an extra syzygy independent of the parity of $\frac{g}{2}$.
\end{remark}

\section{The failure of the Prym-Green Conjecture on $\cR_{8}$}

We argue that the Prym-Green Conjecture A fails on $\cR_8$. Our findings suggest that for a general Prym canonical curve  $\phi_{K_C\otimes \eta}:C\rightarrow \PP^6$ of genus $8$, the non-vanishing
$K_{2, 1}(C, K_C\otimes \eta)\neq 0$ holds and the corresponding  Betti table is the following:
\begin{verbatim}
        0 1  2  3  4 5
 total: 1 8 36 56 35 8
     0: 1 .  .  .  . .
     1: . 7  1  .  . .
     2: . 1 35 56 35 8
\end{verbatim}
Moreover, in all cases the extra syzygy in $K_{2, 1}(C, K_C\otimes \eta)$ is never of maximal rank $7$. We analyze this situation geometrically.

Throughout this section $\ell=2$ and we set $\cR_g:=\cR_{g, 2}$. We recall that $\mathcal{GP}_{g, d}^r$ is the closure in $\cM_g$
of the locus of curves with a base point free linear series $L\in W^r_d(C)$ such that the Petri map $\mu_0(L): H^0(C, L)\otimes H^0(C, K_C\otimes L^{\vee})\rightarrow H^0(C, K_C)$ is not injective. The syzygy locus $\cZ_{8, 2}\subset \cR_8$ considered in Theorem \ref{prymgreenclass} can be extended to the level of the
universal genus $8$ Jacobian $\mathfrak{Pic}_8^{14}\rightarrow \cM_8$ and we introduce the Koszul locus
$$\mathfrak{Kosz}:=\bigl\{[C, L]\in \mathfrak{Pic}_8^{14}: K_{2, 1}(C, L)\neq 0\bigr\}.$$
Note that $\mbox{dim } K_{2, 1}(C, L)=\mbox{dim } K_{1, 2}(C, L)$, for each $[C, L]\in \mathfrak{Pic}_8^{14}$.
It is proved in \cite{Ve} that $\mathfrak{Kosz}$ is indeed a divisor on $\mathfrak{Pic}^{14}_8$, that is, $\mathfrak{Kosz}\neq \mathfrak{Pic}^{14}_8$.
Via the map $\iota:\cR_8\hookrightarrow \mathfrak{Pic}^{14}_8$ given by $\iota([C, \eta]):=[C, K_C\otimes \eta]$, one has that
$\iota^*(\mathfrak{Kosz})=\cZ_{8, 2}$.
\vskip 3pt

Let us consider the parameter space $\Sigma$ of triples $(C, L, V)$, where $[C, L]\in \mathfrak{Pic}^{14}_8$  induces a map $\phi_L:C\rightarrow \PP^6$ and $V\in G\bigl(5, H^0(\PP^6, \I_{C/\PP^6}(2))\bigr)$. The variety $\Sigma$ is birational to a Grassmann bundle over an open subset of $\mathfrak{Pic}^{14}_8$. Denoting by $X:=\mathrm{Bs} |V|\subset \PP^6$ the base locus of the system of quadrics in $V$, one can write that
$X=C+C'$, where the linked curve $C'\subset \PP^{6}$ has genus $14$, degree $\mbox{deg}(C')=18$ and $h^1(C', \OO_{C'}(1))=2$. For a general $(C, L, V)\in \Sigma$, we have an isomorphism
\begin{equation}\label{resi}
H^0(\PP^6, \I_{C/\PP^6}(2))/V\stackrel{\cong}\longrightarrow H^0(C', K_{C'}(-1)).
\end{equation}
This linkage construction induces  a dominant rational map
$$\varphi:\Sigma\dashrightarrow \cM_{14}, \  \ (C, L, \Sigma)\mapsto [C'].$$
The image of $\mathfrak{Kosz}$ under this map has a transparent geometric interpretation. It is proved in  \cite{Ve} Lemma 4.4 that the Petri map for the genus $14$ curve, that is,
$$\mu_0(\OO_{C'}(1)): H^0(C', \OO_{C'}(1))\otimes H^0(C', K_{C'}(-1))\rightarrow H^0(C', K_{C'})$$
is an isomorphism if and only if  $K_{2, 1}(C, L)=0$. We summarize the situation as follows:
\begin{proposition} A paracanonical genus $8$ curve $C\subset \PP^6$ belongs to the Koszul divisor if and only if the residual genus $14$ curve $C'\subset \PP^6$ is Petri special.
\end{proposition}
The locus of such Petri special curves $C'\subset \PP^6$ splits into two components depending on whether $K_{C'}(-1)\in W^1_8(C')$ is base point free, and then $[C']\in \mathcal{GP}_{14, 8}^1$, or else, it has a base point, in which case $C'$ is $7$-gonal.
\vskip 2pt
We fix a point $[C, L]\in \mathfrak{Kosz}$ corresponding to an embedded curve $\phi_L:C\hookrightarrow \PP^6$ of degree $14$ having Betti diagram as above, where $S$ denote the homogeneous coordinate ring of $\PP^6$. Consider the second syzygy matrix $S^7(-2)\oplus S(-3) \leftarrow S(-3)\oplus S^{35}(-4)$ of $ S/I_C$. The entries  of the submatrix $S^7(-2) \leftarrow S(-3)$ are given by linear forms $(\ell_0, \ldots, \ell_6)$ corresponding to a syzygy
$$0\neq \gamma:=\sum_{i=0}^6 \ell_i\otimes q_i\in \mathrm{Ker}\bigl\{H^0(\PP^6, \OO_{\PP^6}(1))\otimes H^0(\PP^6, \I_{C/\PP^6}(2))\rightarrow H^0(\PP^6, \I_{C/\PP^6}(3))\bigr\}.$$
\begin{definition}
The \emph{syzygy scheme} $\mbox{Syz}(\gamma)$ of $\gamma \in K_{2, 1}(C, L)$ is the largest subscheme $Y\subset \PP^6$ such that
$\gamma\in H^0(\PP^6, \OO_{\PP^6}(1))\otimes I_Y(2)$.
The \emph{rank} of $\gamma$ is the dimension of the linear subspace $\langle \ell_0,\ldots, \ell_6 \rangle \subset H^0(\PP^6, \OO_{\PP^6}(1))$.
\end{definition}

We refer to \cite{AN} and \cite{vB} for general background on syzygy schemes.
In case of rank $6$ first order syzygies among quadrics, say $\ell_0=0$, there exists a $6\times 6$ skew-symmetric matrix of linear forms $(a_{ij})_{i, j=1}^6$ such that $q_i=\sum_{j=1}^6 \ell_j a_{ji}$ for $i=1, \ldots, 6$ by \cite[Lemma 4.3]{S}. In this case the syzygy scheme is $C'=\PP^6\cap X_6$, where $X_6\subset \PP^{20}:=\PP_{\ell_i, a_{ij}}$ is the universal rank $6$ syzygy scheme, given by equations $\{q_i=0\}_{i=1}^6$ and $\mathrm{Pfaff}\bigl((a_{ij}\bigr))=0$.
\vskip 3pt

We consider first the case when $\gamma\in K_{2, 1}(C, L)$ is a rank $7$ syzygy, which we can view as a section of the vector bundle $\Omega_{\PP^6}^1(3)$. By direct calculation using the Euler sequence, one finds $c_6(\Omega_{\PP^6}^1(3))=43$, which is the number of zeroes of a general section of $\Omega_{\PP^6}^1(3)$. Since $\gamma\in H^0(\PP^6, \Omega_{\PP^6}^1(3)\otimes \I_C)$, to account for the contribution of $C$ we use the excess intersection formula and from $c_6(\Omega_{\PP^6}^1(3))$ we subtract
$$c_1\bigl(\Omega^1_{\PP^6}(3)_{|C}-N_{C/\PP^6}\bigr)=\mbox{deg}(M_L\otimes L^{\otimes 2})-\mbox{deg}(N_{C/\PP^6})=11\cdot 14-8\cdot 14=42.$$
This shows that $Z(\gamma)=C\cup \{p\}$, where $p\in \PP^6-C$. Choose now a $5$-dimensional subspace $V\subset H^0(\PP^6, \I_{C/\PP^6}(2))$ and let $C'\subset \PP^6$ be the corresponding linked curve. Since all the quadrics in $H^0(\PP^6, \I_{C/\PP^6}(2))$ vanish at $p$, we find via the isomorphism (\ref{resi}) that the pencil $K_{C'}(-1)\in W^1_8(C')$ has a base point at $C'$. In conclusion, via linkage, to rank $7$ syzygies $\gamma\in K_{2, 1}(C, L)$ correspond $7$-gonal curves $C'$ of genus $14$.
\vskip 3pt

Assume now that $\gamma$ is  a rank $6$ syzygy.
Let $p \in \PP^6$ denote the point defined by the $6$ forms. The syzygy $\gamma$ may be interpreted
as a section in $H^0(\PP^6, \F(3))$, where
 $\F$ is the first syzygy sheaf of the ideal sheaf $\I_{p/\PP^6}$:
 $$ 0 \to \F \to \OO^{\oplus 6}_{\PP^6}(-1) \to \I_{p/\PP^6} \to 0.$$
 Since $\mbox{rk}(\F)=5$, the scheme $\mbox{Syz}(\gamma)$ is expected to be $1$-dimensional away from $p$.
 Indeed, the zero locus of a general global section of $\F(3)$ vanishes in $p$ and a smooth half canonically embedded curve $C'$ of degree 21 and genus 22, see \cite[Theorem 4.4]{ELMS}. We refer to \cite{ES} for more details.

 In our case $C'$ is reducible, because $C$ is a component of $C'$. In fact
 $$C' = C \cup E,$$ where $E$ is an elliptic normal curve of degree $7$, intersecting $C$ in $14$ points. Moreover,
 $$C\cdot E\in |\OO_E(2)| \ \mbox{ and }  C\cdot E\in |L^{\otimes 2}\otimes K_C^{\vee}|.$$
 Observe that $[C, L]\in \cR_8$ if and only if $E\cdot C\in |K_C|$. Reversing this construction, we obtain the following result.

 \begin{theorem} The divisor $\mathfrak{Kosz}\subset \mathfrak{Pic}_8^{14}$ is reducible and has a unirational component $\mathfrak{Kosz}_6$ whose general point is a paracanonical curve with extra syzygies of rank $6$, as well as a component $\mathfrak{Kosz}_7$ whose general point is a  curve with syzygies of rank $7$. The syzygy scheme corresponding to a general point $[C, L]\in \mathfrak{Kosz}_6$ is the disjoint union of a point $p\in \PP^6$ and a nodal half-canonical reducible curve $C \cup E$ of genus 22, where $E$ is a smooth elliptic  normal curve of degree $7$ which intersects $C$ in $14$ distinct points.
 \end{theorem}

 \begin{proof} In order to construct $\mathfrak{Kosz}_6$, we reverse the construction and start with an elliptic normal curve $E \subset \PP^6$ and a point $p\in \PP^6-E$.
 The minimal free resolution of $E$ has the following Betti table:
\begin{verbatim}
            0  1  2  3  4 5
     total: 1 14 35 35 14 1
         0: 1  .  .  .  . .
         1: . 14 35 35 14 .
         2: .  .  .  .  . 1
\end{verbatim}

The group $K_{1, 2}(E, \OO_E(1))$ is $35$-dimensional and there is a $22$-dimensional vector space $H^0\bigl(\PP^6, \Omega_{\PP^6}^1(3)\otimes \I_{E\cup \{p\}}\bigr)$ of second syzygies, whose order ideal vanishes at the point $p$. The zero locus of such a syzygy consists of $p$ together with a half-canonical curve, having  $E$ as one of the components. Its residual curve
 is the desired point  $[C, \OO_C(1)] \in \mathfrak{Kosz}$. To show that such curves fill-up a component of $\mathfrak{Kosz}$, we count dimensions. Modulo the action of $PGL(7)$, septic elliptic curves $E\subset \PP^6$ depend on one parameter. Indeed, if $H_7\subset SL_7(\mathbb C)$ denotes the level $7$
Heisenberg group, there is a  $1:1$ correspondence between elliptic curves with a level $7$ structure and $H_7$-equivariantly embedded normal elliptic curves $E\subset \PP^6$.  The choice of $p \in \PP^6$ gives $6$ dimensions. Having chosen $E$ and $p$ as above, the choice $[\gamma] \in \PP\bigl(H^0(\PP^6, \Omega^1_{\PP^6}(3)\otimes \I_{E\cup \{p\}})\bigr)$ gives another $21$ dimensions. Since the construction depends on $1+6+21=\mbox{dim}(\mathfrak{Pic}_8^{14})-1$ parameters, curves obtained in this way fill-up a component $\mathfrak{Kosz}_6$ of $\mathfrak{Kosz}$. To complete the proof of the unirationality of $\mathfrak{Kosz}_6$, it remains to show that the construction leads to a smooth curve for general choices of parameters. By semi-continuity this can checked with {\it Macaulay2} over a finite field, for details see the function
\emph{\href{http://www.math.uni-sb.de/ag/schreyer/home/M2/doc/Macaulay2/KoszulDivisorOnPic14M8/html/_unirationality__Of__D1.html}{unirationaliyOfD1}} in our {\it Macaulay2} package
\emph{\href{http://www.math.uni-sb.de/ag/schreyer/home/M2/doc/Macaulay2/KoszulDivisorOnPic14M8/html/index.html}{KoszulDivisorOnPic14M8}}.

\vskip 3pt
By running the  function  \emph{\href{http://www.math.uni-sb.de/ag/schreyer/home/M2/doc/Macaulay2/KoszulDivisorOnPic14M8/html/_get__Curve__On__Koszul__Divisor.html}
{getCurveOnKoszulDivisor}} contained in the {\it Macaulay2} package \emph{\href{http://www.math.uni-sb.de/ag/schreyer/home/M2/doc/Macaulay2/KoszulDivisorOnPic14M8/html/index.html}{KoszulDivisorOnPic14M8}}, we observe that there exist points $[C, L]\in \mathfrak{Kosz}$ with syzygies $0\neq \gamma\in K_{1, 2}(C, L)$ having full rank $7$. They cannot lie in the closure of $\mathfrak{Kosz}_6$ for the rank of a syzygy is upper semi-continuous, therefore $\mathfrak{Kosz}_6\subsetneq \mathfrak{Kosz}$.
\end{proof}
\begin{remark}
Since $\mathfrak{Pic}^{14}_8$ is unirational \cite{M}, \cite{Ve}, over a finite field $\FF_p$ one can find points on absolutely irreducible components of $\mathfrak{Kosz}$ with probability approximately equal to $\frac{1}{p}$. By our \emph{\href{http://www.math.uni-sb.de/ag/schreyer/home/M2/doc/Macaulay2/KoszulDivisorOnPic14M8/html/_experiment1.html}{experiment1}} of the package  \emph{\href{http://www.math.uni-sb.de/ag/schreyer/home/M2/doc/Macaulay2/KoszulDivisorOnPic14M8/html/index.html}{KoszulDivisorOnPic14M8}}, we observe that the rate of points
$[C, L]\in \mathfrak{Kosz}$ corresponding to syzygies of rank $6$ respectively $7$ is approximately the same.  We conclude that most likely, $\mathfrak{Kosz}$ has precisely two components, namely $\mathfrak{Kosz}_6$ and $\mathfrak{Kosz}_7$.
\end{remark}

The remaining part of this section is devoted to obtain strong evidence for  the inclusion $\cR_8\subset \mathfrak{Kosz}$. Since we do not have a unirational parametrization of $\cR_8$ (however see \cite{FV3} for results in this direction), we are only able to perform {\it Macaulay2} experiments with smooth curves from lower dimensional subvarieties
 of $\rr_8$, for instance, we tested the curves constructed by the function
 \emph{\href{http://www.math.uni-sb.de/ag/schreyer/home/M2/doc/Macaulay2/PrymCanonicalCurves/html/_get__Canonical__Curve__Of__Genus8__With2__Torsion.html}
  {getCanonicalCurveOfGenus8With2Torsion}} of our package \emph{ \href{http://www.math.uni-sb.de/ag/schreyer/home/M2/doc/Macaulay2/PrymCanonicalCurves/html}{PrymCanonicalCurves}}.
 All these calculations lead to curves in the component $\mathfrak{Kosz}_6$, having an extra syzygy of rank $6$. We first observe that in order to show the inclusion $\cR_8\subset \mathfrak{Kosz}$, it suffices to show that general $1$-nodal curves from the boundary divisor $\Delta_0^{'}\subset \rr_8$ lie in the locus $\zz_{8, 2}$.
 \begin{proposition}\label{delta0}
 Assuming that the inclusion $\Delta_0^{'}\subset \zz_{8, 2}$ holds, then $\zz_{8, 2}=\rr_8$ and the Prym-Green Conjecture fails on $\cR_8$.
 \end{proposition}
 \begin{proof}
 We proceed by contradiction and assume that $\zz_{8, 2}$ is a divisor on $\rr_8$ whose class is computed in Theorem \ref{prymgreenclass}. Precisely, we have the relation
 $$[\zz_{8, 2}]^{\mathrm{virt}}-[\zz_{8, 2}]=27\lambda-4(\delta_0^{'}+\delta_0^{''})-6\delta_0^{(1)}-[\zz_{8, 2}]\in \mathbb Q_{\geq 0}\langle \delta_0^{'}, \delta_0^{''}, \delta_0^{(1)}\rangle.$$
 First we observe that the class $[\zz_{8, 2}]^{\mathrm{virt}}-16\delta_{0}^{''}$ is effective, that is, the morphism of vector bundles
 $\varphi_{1, 2}:\mathbb H_{1, 2}\rightarrow \mathbb G_{1, 2}$ over the stack $\wt{\mathsf{R}}_{8, 2}$ is degenerate with multiplicity at least $16$ along the boundary divisor $\Delta_0^{''}$. Indeed, assume that $[X, \eta, \phi]\in \Delta_0^{''}$ is a general point corresponding to a normalization $\mathrm{nor}:C\rightarrow X$, where $[C]\in \cM_7$. Since $\mathrm{nor}^*(\eta)=\OO_C$, it quickly follows that $K_{1, 2}(X, \omega_X\otimes \eta)=K_{1, 2}(C, K_C)$ and the latter space is $16$-dimensional. This shows that the class $c_1(\mathbb{G}_{1, 2}-\mathbb{H}_{1, 2})-16\delta_0^{''}$ is effective.
 \vskip 3pt
 Assume now that $\Delta_0^{'}\subset \zz_{8, 2}$, therefore $[\zz_{8, 2}]^{\mathrm{virt}}-\delta_0^{'}-16\delta_0^{''}=27\lambda-5\delta_0^{'}-20\delta_{0}^{''}-6\delta_{0}^{(1)}$ is still effective.
Combining this class with that of the pull-back of the Brill-Noether divisor $[\mm_{8, 7}^2]=22\lambda-3\delta_0\in \mbox{Eff}(\wt{\mathsf{M}}_8)$, see \cite{EH}, after routine manipulations we can form an effective representative of the canonical class $K_{\rr_8}$. This is a contradiction, since $\rr_8$ is uniruled, see \cite{FV3}.
 \end{proof}

Even though we do not know whether $\Delta_0^{'}$ is unirational,  we quote  from \cite{FV3}:

\begin{theorem}\cite{FV3}\label{uniratR7}
The Prym moduli space $\cR_{7}$ is unirational.
\end{theorem}

\noindent {\it Sketch of proof.} A general Prym canonical curve $C \subset \PP^5$ of genus $7$  has  the Betti table
 \begin{verbatim}
            0  1  2  3 4
     total: 1 11 27 24 7
         0: 1  .  .  . .
         1: .  3  .  . .
         2: .  8 27 24 7
\end{verbatim}
The quadrics in $H^0(\PP^5, \I_{C/\PP^5}(2))$ intersect in a \emph{Nikulin surface} $Y\subset \PP^5$, that is, a smooth $K3$ surface containing 8 disjoint lines $L_1, \ldots, L_8$ with $L_i^2=-2$ and $C\cdot L_i=0$, for $i=1, \ldots, 8$. Furthermore, $2(C-H)\equiv L_1+\cdots+L_8$. The linear system
$$\bigl|\OO_Y(C - L_1 -\ldots -L_7)\bigr|$$ is zero dimensional and consists of a single rational normal quintic curve $R_5\subset \PP^5$. The unirational parametrization of $\cR_{7}$ reverses this construction. One starts with a fixed rational normal curve $R_5 \subset \PP^5$  together with seven general secant lines $L_1, \ldots, L_8$. The union $R_5 \cup L_1 \cup \ldots \cup L_7$ is a $14$-nodal stable curve which lies on three quadrics intersecting in a smooth Nikulin surface $Y$. A general $C \in |\OO_Y(R_5+L_1+\ldots+L_7)|$ is a general curve in $\cR_{7}$. This establishes a dominant rational map from a $\PP^7$-bundle over $\mm_{0, 14}$ onto $\cR_{7}$. In the function \emph{\href{http://www.math.uni-sb.de/ag/schreyer/home/M2/doc/Macaulay2/PrymCanonicalCurves/html/_random_Prym_canonical_curve_Of_Genus7.html}
{randomPrymCanonicalCurveOfGenus7}} of our package \emph{\href{http://www.math.uni-sb.de/ag/schreyer/home/M2/doc/Macaulay2/PrymCanonicalCurves/html}{PrymCanonicalCurves}} you find  an "implementation" of the unirational parametrization.
 \qed

 Based on Theorem \ref{uniratR7} we have strong evidence that the inclusion $\Delta_0' \subset \cR_{8} \cap \cZ_1$ holds, therefore also for the equality
 $\zz_{8, 2}=\rr_8$. We start with a randomly chosen curve in $C \in \cR_{7}(\FF_p)$ for a prime $p$ with $10^4 < p <3 \cdot 10^4$.
 By the Hasse-Weil theorem, one has the following estimate for the number of rational points of a curve $C$ over $\FF_p$:
 $$ p+1-2g\sqrt p\leq |C(\FF_p)| \leq  p+1 + 2g \sqrt p.$$
 The rational points are easy to find: We consider a plane model  $C_{\mathrm{pl}} \subset \PP^2$ of $C$ defined over $\FF_p$. Approximately $1-\frac{1}{e} \approx 63 \% $
 of the lines $L \in (\PP^2)^{\vee}(\FF_p)$ intersect $C$ in at least one $\FF_p$-rational point. We find the $\FF_p$-rational intersection points by computing the primary decomposition of
 $I_{C_{\mathrm{pl}}} + I_L \subset \FF_p[x,y,z] $. A probabilistic algorithm gives us two points $P, Q \in C(\FF_p)$. Consider now the curve
 $[C'=\frac{C}{P\sim Q}] \in \Delta_0 \subset \mm_8$ obtained by identifying $P$ and $Q$, and let $\mathrm{nor}: C \to C'$ be the normalization map.
 We now apply the following:

 \begin{proposition} Let $\mathrm{nor}: C \to C'$  be the partial normalization of a single node of a irreducible nodal curve $C'$ of genus $g$ defined over a field $\FF$ and let $P,Q \in C(\FF)$ be the preimages of the node of $C'$. Let $\eta=\OO_C(D_1-D_2) \in \mathrm{Pic}^0 C [2]$ be a $2$-torsion line bundle on $C$ curve with $D_1,D_2$ being effective divisors defined over $\FF$
 with support disjoint from $\{P,Q\}$.
Then there exists a two torsion line bundle $\eta' \in  \mathrm{Pic}^0 C [2]$ defined over $\FF$ with $\mathrm{nor}^* \eta' \isom \eta$ if and only if $\frac{f(P )}{f(Q)} \in (\FF^*)^2$ where
 $f \in \FF(C )$ is the  rational function with $(f)=2D_1-2D_2$.
 \end{proposition}

\begin{proof} Consider an embedding $C \subset \PP^r$ defined over $\FF$, e.g. a (pluri)canonical model. Suppose $\frac{f(P )}{f(Q)} = \alpha^2$. Consider an arbitrary rational function $g=\frac{g_0}{g_1} \in \FF(C )$, where $g_0,g_ 1$ are linear forms on $\PP^r$. We choose a matrix
$$\begin{pmatrix}
a&b \cr
c & d \cr
\end{pmatrix} \in \SL(2,\FF)$$
such that the rational function $h=\frac{a g_0+bg_1}{cg_0+dg_1}$ takes values $h(P )=1$ and $h(Q ) =\alpha$.
Let $ (h) = (h)_0-(h)_\infty$ be the principal divisor of $h$ on $C$. The divisor $E_1-E_2$ with $E_1=D_1+(h)_0$ and $E_2=D_2+(h)_\infty$ is
$2$-torsion on $C'$
since $\frac{(fh)(P )}{(fh)(Q )} =1$ and $\mathrm{nor}^*\OO_{C'}(E_1-E_2)=\OO_C(D_1-D_2)$,  because  $(E_1-E_2) -(D_1-D_2)$ is a principal divisor on $C$.
This proves that the condition is sufficient. To establish necessity,  we note that $1 \in (\FF^*)^2$
\end{proof}

Thus for a finite field of $\ch(\FF_p) \not=2$, in about $50 \%$ of the cases we can descend the $2$-torsion bundle from $C$ to $C'$. In this way we find many $\FF_p$-rational points on $\Delta_0'$. Computing the syzygies of the resulting Prym canonical curves we always land in the component $\mathfrak{Kosz}_6$, see
 \emph{\href{http://www.math.uni-sb.de/ag/schreyer/home/M2/doc/Macaulay2/PrymCanonicalCurves/html/random__One__Nodal__Prym_
 _Canonical__Curve__Of__Genus8.html}{randomOneNodalPrymCanonicalCurveOfGenus8}} in our package
 \emph{\href{http://www.math.uni-sb.de/ag/schreyer/home/M2/doc/Macaulay2/PrymCanonicalCurves/html/}{PrymCanonicalCurves}}.
This does not establish the inclusion $\Delta_0' \subset \mathfrak{Kosz}_6$ (hence the equality $\cZ_{g, 2}=\cR_8$), since for any fixed prime $p$, we only have  finitely many points in $\Delta_0'(\FF_p)$. However, the possibility that $\Delta_0' \not\subset \mathfrak{Kosz}_6$ appears to be exceedingly unlikely.


\begin{thebibliography}{EMS}
\bibitem[ACV]{ACV} D. Abramovich, A. Corti and A. Vistoli, {\em{Twisted bundles and admissible coverings}},
 Communications in  Algebra \textbf{31} (2003), 3547-3618.


\bibitem[AJ]{AJ}
{D.~Abramovich and T.~J.~Jarvis}, {\em{Moduli of twisted spin
curves}}, Proc. Amer. Math. Soc. {\bf131} (2003), 685-699,

\bibitem[AN]{AN}M. Aprodu and J. Nagel,
{\em{Koszul cohomology and algebraic geometry}},
University Lecture Series, Vol. 52, American Mathematical Society 2010.

\bibitem[AV]{AV} {D.~Abramovich and A.~Vistoli.}
{\em{Compactifying the space of stable maps.}} J.~Amer.~Math.~Soc.
{\bf 15}  (2002), 27-75.

\bibitem[BC]{BC} I. Bauer and F. Catanese, {\em{The rationality of certain moduli spaces of curves of genus $3$}}, in: Cohomological and geometric approaches to rationality problems, Progress in  Math. 282, Birkh\"auser 2010.
\bibitem[BV]{BV} I. Bauer and A. Verra, {\em{The rationality of the moduli space of genus-4 curves endowed with an order-3 subgroup of their Jacobian}}, Michigan Journal of Math. \textbf{59} (2010), 483-504.
\bibitem[B]{B} A. Beauville, {\em{Fibr\'es de rang $2$ sur une courbe, fibr\'e d\'eterminant et fonctions th\^eta}}, Bull. Soc. math. France \textbf{116} (1988), 431-448.
\bibitem[CCC]{CCC} L. Caporaso, C. Casagrande, and M. Cornalba, {\em{Moduli of roots of
line bundles on curves}}, Transactions of the American Math. Society
\textbf{359} (2007), 3733-3768.
\bibitem[C1]{C} A. Chiodo, {\em{Stable twisted curves and their $r$-spin structures}}, Annales de l'institut Fourier, \textbf{58} (2008), 1635-1689.
\bibitem[C2]{CGRR} {A.~Chiodo,}
{\em{Towards an enumerative geometry of the moduli space of
twisted curves and $r$th roots}}, Compositio Math. \textbf{144} (2008),  1461-1496.
\bibitem[CF]{CF} A. Chiodo and G. Farkas, {\em{Singularities of the moduli space of level curves}}, arXiv:1205.0201.
\bibitem[DGP]{DGP} J.-G. Dumas, P. Giorgi and C. Pernet,  {\em{Dense Linear Algebra over Finite Fields}}, ACM Transactions on Mathematical Software, \textbf{35}, 3 (2009), 34 pp.

\bibitem[E]{E} D. Eisenbud, {\em{Commutative algebra}}, Graduate Texts in Mathematics \textbf{150}, Springer Verlag 2004.
\bibitem[EH]{EH} D. Eisenbud and J. Harris, {\em{The Kodaira
dimension of the moduli space of curves of genus $\geq 23$}},
Inventiones Math. \textbf{90} (1987), 359-387.
\bibitem[ELMS]{ELMS}
D. Eisenbud, H. Lange, G. Martens and F.-O. Schreyer, {\em{The Clifford dimension of a projective curve}},
Compositio Math. \textbf{72} (1989), 173-204.
\bibitem[ES]{ES} F. Eusen and F.-O. Schreyer, {\em{A remark on a conjecture of Paranjape and Ramanan}}, in: Geometry and Arithmetic, Schiermonnikoog 2010 (C. Faber, G. Farkas, R. de Jong eds), 113-123, EMS Series of Congress Reports 2012.
\bibitem[F1]{F1} G. Farkas, {\em{Koszul divisors on moduli spaces of
curves}}, American Journal of Math. \textbf{131} (2009), 819-869.
\bibitem[FL]{FL} G. Farkas and K. Ludwig, {\em{The Kodaira dimension of the moduli space of Prym varieties}}, Journal of the European Math. Society \textbf{12} (2010), 755-795.
\bibitem[FMP]{FMP} G. Farkas, M. Musta\c{t}\u{a} and M. Popa, {\em{Divisors on
$\cM_{g, g+1}$ and the Minimal Resolution Conjecture for points on
canonical curves}}, Annales Scientifique de L\'Ecole  Normale
Sup\'erieure \textbf{36} (2003), 553-581.
\bibitem[FP]{FP} G. Farkas and M. Popa, {\em{Effective divisors on
$\mm_g$, curves on $K3$ surfaces and the Slope Conjecture}}, Journal
of Algebraic Geometry \textbf{14} (2005), 151-174.
\bibitem[FV1]{FV1} G. Farkas and A. Verra, {\em{The classification of the universal Jacobian over the
moduli space of curves}}, arXiv:1005.5354, to appear in Commentarii Math. Helvetici.
\bibitem[FV2]{FV2} G. Farkas and A. Verra, {\em{Moduli of theta characteristics via Nikulin surfaces}}, Mathematische Annalen \textbf{354} (2012), 465-496.
\bibitem[FV3]{FV3} G. Farkas and A. Verra, {\em{$\cR_8$ is uniruled}}, preprint.

\bibitem[GS]{GS} D. Grayson and M. Stillman, {\em{\it Macaulay2}, a software system for research
                   in algebraic geometry}, \href{http://www.math.uiuc.edu/Macaulay2/}{http://www.math.uiuc.edu/Macaulay2/}

\bibitem[G]{G} M. Green, {\em{Koszul cohomology and the cohomology of projective varieties}}, Journal of Differential Geometry \textbf{19} (1984), 125-171.
\bibitem[HM]{HM} J. Harris and D. Mumford, {\em{On the Kodaira
dimension of $\mm_g$}}, Inventiones Math. \textbf{67} (1982), 23-88.
\bibitem[J1]{J} T.~Jarvis, {\em{Geometry of the moduli of higher spin curves}}, International J. of Math.,
\textbf{11} (2000), 637-663.
\bibitem[J2]{JPic} T.~Jarvis, {\em{The Picard group of the moduli of higher spin curves}}, New York J. of Math. {\bf7} (2001), 23-47.
\bibitem[M]{M} S. Mukai, {\em{Curves and Grassmannians}}, in: Algebraic Geometry and Related Topics, Inchon 1992,  (J.-H. Yang, Y. Namikawa, K. Ueno, eds),  1992, 19-40, International Press.
\bibitem[S]{S} F.-O.~Schreyer,
{\em{ A standard basis approach to syzygies of canonical curves}},
J. Reine Angew. Math. \textbf{421} (1991), 83-123.

\bibitem[vB]{vB} H.-C. Graf v. Bothmer,  {\em{
Generic syzygy schemes}},
J. Pure Appl. Algebra \textbf{208} (2007),  867-876.
\bibitem[Ve]{Ve} A. Verra, {\em{The unirationality of the moduli
space of curves of genus $14$ or lower}}, Compositio Math. \textbf{141} (2005), 1425-1444.

\bibitem[V]{V}
C. Voisin, {\em{Green's canonical syzygy conjecture for generic curves of odd genus}},
Compositio Math. \textbf{141} (2005), 1163-1190.
\end{thebibliography}
\end{document}